\documentclass[]{amsart}

\usepackage{cite}

\makeatletter
\def\@tocline#1#2#3#4#5#6#7{\relax
  \ifnum #1>\c@tocdepth %
  \else
    \par \addpenalty\@secpenalty\addvspace{#2}%
    \begingroup \hyphenpenalty\@M
    \@ifempty{#4}{%
      \@tempdima\csname r@tocindent\number#1\endcsname\relax
    }{%
      \@tempdima#4\relax
    }%
    \parindent\z@ \leftskip#3\relax \advance\leftskip\@tempdima\relax
    \rightskip\@pnumwidth plus4em \parfillskip-\@pnumwidth
    #5\leavevmode\hskip-\@tempdima
      \ifcase #1
       \or\or \hskip 1em \or \hskip 2em \else \hskip 3em \fi%
      #6\nobreak\relax
    \hfill\hbox to\@pnumwidth{\@tocpagenum{#7}}\par%
    \nobreak
    \endgroup
  \fi}
\makeatother

\usepackage{amsmath}
\usepackage{amsfonts}
\usepackage{amssymb}
\usepackage{amsthm}
\usepackage{mathtools}
\usepackage{enumitem}
\usepackage{bm}
\usepackage{hyperref}
\usepackage{tikz}
\usetikzlibrary{positioning}%
\usetikzlibrary{arrows}%
\usepackage{tikz-cd}
\usepackage[normalem]{ulem}
\usepackage{rotating}
\usepackage{color}
\usepackage{cleveref}
\usepackage{comment}
\DeclareRobustCommand{\SkipTocEntry}[5]{}

\newtheorem{theorem}{Theorem}
\newtheorem{prop}[theorem]{Proposition}

\newtheorem{corollary}[theorem]{Corollary}

\newtheorem{lemma}[theorem]{Lemma}

\newtheorem{problem}[theorem]{Problem}
\newtheorem{claim}[theorem]{Claim}

\theoremstyle{definition}
\newtheorem{definition}[theorem]{Definition}
\newtheorem{example}[theorem]{Example}

\theoremstyle{remark}
\newtheorem{remark}[theorem]{Remark}
\numberwithin{theorem}{subsection}
\numberwithin{equation}{section}

\theoremstyle{citeplain}

\DeclareMathOperator{\tr}{tr}

\DeclareMathOperator{\rk}{rk}

\DeclareMathOperator{\Hom}{Hom}

\DeclareMathOperator{\Sym}{Sym}
\DeclareMathOperator{\Spec}{Spec}
\DeclareMathOperator{\Proj}{Proj}
\DeclareMathOperator{\grad}{grad}

\DeclareMathOperator{\adj}{adj}
\DeclareMathOperator{\Pic}{Pic}
\DeclareMathOperator{\pf}{pf}

\def\CC{\mathbb{C}}
\def\QQ{\mathbb{Q}}
\def\RR{\mathbb{R}}
\def\ZZ{\mathbb{Z}}
\def\PP{\mathbb{P}}
\def\NN{\mathbb{N}}

\def\GG{\mathbb{G}}
\def\Z+{\mathbb{Z}_{\geq 0}}
\def\R+{\mathbb{R}_{\geq 0}}

\def\ot{\otimes}

\def\hook{\lrcorner}
\def\CQ{\mathcal{CQ}} %
\def\CS{\mathcal{CS}}
\def\ComplCol{\mathcal{CC}}
\def\Per{P}
\def\PPer{\mathcal{P}}
\def\fL{\mathcal{L}}
\def\fQ{\mathcal{Q}}

\def\fU{\mathcal{U}}
\def\ot{\otimes}

\newcommand\restr[2]{{%
		\left.\kern-\nulldelimiterspace %
		#1 %
		\vphantom{\big|} %
		\right|_{#2} %
}}

\title[Applications of intersection theory]{Applications of intersection theory:\\
from maximum likelihood to chromatic polynomials}
\author{Rodica Andreea Dinu, Mateusz Micha{\l}ek and Tim Seynnaeve}
\begin{document}
\begin{abstract}
Recently, we have witnessed tremendous applications of algebraic intersection theory to branches of mathematics, that previously seemed very distant. In this article we review some of them.
Our aim is to provide a unified approach to the results e.g.~in the theory of chromatic polynomials (work of Adiprasito, Huh, Katz), maximum likelihood degree in algebraic statistics (Drton, Manivel, Monin, Sturmfels, Uhler, Wi\'sniewski),  Euler characteristics of determinental varieties (Dimca, Papadima), characteristic numbers (Aluffi, Schubert, Vakil) and the degree of semidefinite programming (Bothmer, Nie, Ranestad, Sturmfels). Our main tools come from intersection theory on special varieties called the varieties of complete forms (De Concini, Procesi, Thaddeus) and the study of Segre classes (Laksov, Lascoux, Pragacz, Thorup).
\end{abstract}

\maketitle
\tableofcontents
\section{Introduction}

In a broad sense the aim of the article is to review applications of intersection theory. From the very beginning we are doomed to fail, as on the one hand there are too many applications even to fit in one book, on the other hand setting up correct foundations of intersection theory in detail is very technical and lengthy. Thus, we focus on very particular, recent applications, especially those that appear in disciplines not classically related to algebraic geometry. The reader interested in more technical aspects is referred to the great books \cite{FultonIntersection, eisenbud20163264}.

Our approach is not very different from what has happened in the history of intersection theory. Most mathematicians first enjoyed the ``applications" as classically done by Schubert, with the technical details only proved later (in fact, throughout the century). One of the reasons is that intersection theory often provides very intuitive, simple tools that give many concrete answers and lead to great results, however proving that the given methods always work may require tremendous work. Nowadays, most of the classical results are well-justified. Thus we hope the reader may enjoy the forthcoming journey, without worrying about the ``engine".

We will present results mainly about two invariants: chromatic polynomials of graphs and maximum likelihood (ML) degrees of linear concentration models \cite{sturmfels2010multivariate}. A priori, these objects are completely unrelated, apart from the fact that both: coefficients of the chromatic polynomial and ML-degrees are integers. The first aim is to interpret those integers in terms of intersection theory. Here, the approach is through projective degrees of rational maps. First we associate to our object (graph or statistical model) a subspace $\PP(L)$ of a projective space $\PP(V)$. The space $V$ depends on the type of the object (e.g.~number of edges of the graph) and the space $L$ precisely encodes the object. Further, the space $V$ comes to us with a distinguished polynomial $F$. The graph of the gradient of $F$ restricted to $\PP(L)$ turns out to be a variety that encodes a lot of the invariants of the object we started from. By performing intersection theory on that graph we are able to recover those invariants.

From the technical point of view, the graph of the gradient of $F$ is in most cases too singular to be able to effectively perform intersection theory on it. Fortunately, there exist resolutions of this graph with wonderful properties. The celebrated example, when our object is a graph (or more generally a matroid), was presented in \cite{huh2012milnor}, where $F$ is simply the product of the coordinates and the resolution is the permutohedral variety. In a greater generality we present various varieties of complete forms, which are precisely the resolutions we need.

We believe that the piece of mathematics presented in the review has advantages exceeding simply being beautiful. First, it gives a new point of view on important invariants. This in many cases allows a simpler computation of the invariants (although using sophisticated techniques!). Second, algebraic geometry provides a wealth of structural theorems, which reveal properties of the invariants, that were hard to see or prove otherwise. Last, using the presented techniques may lead to new invariants of the classical objects, that might be used e.g.~in classification problems.

The interplay of algebraic intersection theory and the study of invariants coming from different disciplines reveals new structure, that seems to be missed before. In classical intersection theory, we study one variety $X$, possibly with distinguished (classes of) subvarieties $Y_1,\dots,Y_k$, obtaining the integer invariants by intersecting the $Y_i$'s. The conjectures, coming e.g.~from algebraic statistics, concern sequences of numbers, when the variety $X$ changes. Still, as we will see, one can often obtain a nice, structured answer, concerning the invariants, by applying methods from enumerative geometry. This shift of interest from studying intersection theory on one $X$, to the study how the answer changes, when $X$ changes, is a part of the \emph{Bodensee program} which we introduce in the last Section~\ref{sec:Bodensee} of the article.

\section*{Acknowledgements}
We would like to thank Thomas Endler from the Max Planck Institute for Mathematics in the Sciences for helping us with the graphics.

\section{Gradients and geometry: general construction}
Let us consider a homogeneous polynomial $F\in S^dV^*$ of degree $d>0$ on the finite dimensional vector space $V$.
It partitions $V$ in sets $V_{F=\lambda}:=\{v\in V: F(v)=\lambda\}$.

\subsection{What is the gradient?}
Let us fix $v\in V$ with $F(v)=\lambda\in\CC$ and assume that $v$ is a smooth point of $V_{F=\lambda}$. The tangent space $T_v V_{F=\lambda}$ at $v$ is a codimension one affine subspace of $V$, which we identify with the (parallel) codimension one vector subspace $T_vF\subset V$ and hence a line $(T_vF)^\perp\subset V^*$. We note that scaling $v$ leads to the same line in the dual space.

We thus obtain a rational map:
\[\grad F:\PP(V)\dashrightarrow \PP(V^*).\]
The map is not defined precisely at those points $[v]\in \PP(V)$, where $v$ is a singular point of  $V_{F=F(v)}$.

In coordinates, i.e.~identifying $V\simeq \CC^{n+1}$, the map $\grad F$ is indeed the gradient of $F$, i.e.~is given by the partial derivatives $\frac{\partial F}{\partial x_0},\dots,\frac{\partial F}{\partial x_n}$. This follows from the fact that the normal vector to the tangent space to a hypersurface $V(F)$ is, by definition, the gradient of $F$, and the partial derivatives do not depend on the shift of $F$ by a constant. However, what is extremely important at this point is that the gradient map \emph{naturally} goes from a projective space to the \emph{dual} projective space. In fact, according to Dieudonn{\'e}: \emph{``Grassmann's greatest tragedy was that he did not make a distinction between a projective space and the dual projective space.''}

As the construction of the gradient map is of absolute importance for our survey, we will present one more, purely algebraic. The set of (fully) symmetric tensors $S^d V^*$ is canonically a linear subspace of partially symmetric tensors $V^*\otimes S^{d-1}V^*$. The linear inclusion:
\begin{equation} \label{eq:partiallySymmetricTensors}
	S^d V^*\hookrightarrow V^*\otimes S^{d-1}V^*
\end{equation}
equivalently gives the linear map:
\[S^{d-1}V \otimes S^dV^*\rightarrow V^*.\]
Hence, we have a bilinear map:
 \[\hook: S^{d-1}V \times S^dV^*\rightarrow V^*.\]
 The projectivisation of the linear map:
 \[\hook F: S^{d-1}V\rightarrow V^*\]
 precomposed with the $(d-1)$-st Veronese map $\PP(V)\rightarrow \PP(S^{d-1}V)$ is the gradient $\grad F$.
 \begin{example}
Let us consider a polynomial $F(x,y,z)=3x^2y+7z^3$ on $\CC^3$. Consider the point $v:=(-1,2,1)\in \CC^3$.

We have:
\[\grad F(-1,2,1)=(6xy,3x^2,21z^2)(-1,2,1)=[-12:3:21]\in\PP^2.\]

The map:
\[S^3(\CC^3)^*\rightarrow (\CC^3)^*\otimes S^2(\CC^3)^*\]
sends (up to a constant):
\[x^2y\mapsto x\otimes xy+x\otimes xy+y\otimes x^2,\quad z^3\mapsto 3 z\otimes z^2.\]
Further the second Veronese of $v$ is $v_2:=1e_1^2+4e_2^2+1e_3^2-4e_1e_2-2e_1e_3+4e_2e_3$. The pairing on $S^2\CC^3$ and $S^2(\CC^3)^*$ is given by $\langle e_1^2,x^2\rangle=2$ and $\langle e_1e_2,xy\rangle=1$. We obtain:
\[v_2\hook F=v_2\hook (6 x\otimes xy+3y\otimes x^2+21 z\otimes z^2)=6x\cdot(-4)+3y\cdot 2+21z\cdot 2.\]
Thus, we get $[-24x+6y+42z]\in \PP((\CC^3)^*)$, which coincides with the direct computation above.
 \end{example}

 Clearly, $\grad F$ has played a prominent role in geometry and algebra for hundreds of years. For example, the restriction of $\grad F$ to the variety $V(F)$ is essentially the Gauss map and the (closure of the) image of this restriction is the dual variety of $V(F)$.

\begin{remark}
In this survey we will focus on the gradient map. This could be generalized in many directions. The first one is to consider higher order derivatives. As we will see this will be often a good idea leading to resolutions of singularities of the graph of the gradient map.

Even more generally, many of the topics we describe could be implemented to rational maps between projective spaces. This however, would take us too far away from our original motivations. Our aim is to obtain nice, computable invariants of algebraic, geometric, or combinatorial objects that appear in many branches of mathematics.
\end{remark}

Of central importance will be the (closure of the) graph $\Gamma$ of our map $\grad F$:
\begin{equation}\label{eq:diagramGradF}
\begin{tikzcd}
& \Gamma\subset\PP(V)\times \PP(V^*) \arrow[ld,"\pi_1"'] \arrow[rd,"\pi_{n-1}"] &  \\
\PP(V) \arrow[rr,"\grad F"',dashed] & & \PP(V^*)
\end{tikzcd}
\end{equation}

Let us now fix a linear subspace $L$ of $V$. We assume that $\PP(L)$ is not contained in the indeterminacy locus of $\grad{F}$ and, in particular, we may consider the restriction $\restr{(\grad F)}{L}$ of the gradient map to $\PP(L)$.
Then we can make the following two constructions.

\subsection{First construction: restriction of the gradient}
 Let $D(\grad(F))\subset \PP(V)$ be the set where $\grad(F)$ is well-defined.
 We can consider the strict transform
 $$ \hat{L} := \overline{\pi_1^{-1}(\PP(L) \cap D(\grad(F)))} \subset \Gamma \subset \PP(V) \times \PP(V^*)$$
 of $\PP(L)$.

 We recall that the two basic invariants of a variety in a projective space are its dimension and degree, which are nonnegative integers \cite{NonlinearAlgebra}. Analogously to the degree, for a variety that is a subvariety of a product of two projective spaces we obtain a vector of numbers, known as \emph{bidegree} or \emph{multidegree}. Explicitly: given a $d$-dimensional variety $X\subset \PP^a\times\PP^b$, for $i=0,\dots, d$ we get the integer $\mu_i$ which is the number of points obtained by intersecting $X$ with $d-i$ general hyperplanes in $\PP^a$ and $i$ general hyperplanes in $\PP^b$ \cite[Section 8.5]{miller2004combinatorial}.

 In the Chow ring/cohomology ring $A(\PP^a \times \PP^b)$, we have
 \[
 [X]=\sum_{i=0}^{d}{\mu_i [\PP^{d-i} \times \PP^{i}]} \in A_{d}(\PP^a \times \PP^b).
 \]
 and
 \begin{equation} \label{eq:muIntersectionProduct}
 \mu_i = \int_{X}{H_1^{d-i} H_2^{i}},
 \end{equation}
 where $H_1 \in A^1(X)$ is the pullback of the hyperplane class in $\PP^a$ under $X \to \PP^a$, and similarly $H_2 \in A^1(X)$ is the pullback of the hyperplane class in $\PP^b$ under $X \to \PP^b$.

\begin{definition}
 The bidegrees of the variety $\hat{L}$ will be denoted by $\mu_i(L)$ for $i=0,\dots,\dim \hat{L}=\dim \PP(L)$. They depend on $F$ and on $L$.
\end{definition}

\emph{Assumption 1}: From now on we assume that the gradient is a generically finite map on $\PP(L)$.

\begin{remark}\label{isgenericof}
	If $\PP(L) \subset \PP(V)$ is generic of dimension $d$, then $\mu_i(L) = \mu_{i}(V)$, as they are both the bidegree of the variety obtained by intersecting $\Gamma$ with $n-i$ general hyperplanes in $\PP(V)$ and $i$ general hyperplanes in $\PP(V^*)$.
\end{remark}

Equivalently, one can consider the graph of the map $\restr{(\grad F)}{L}$:
\begin{equation*}
\begin{tikzcd}
& \Gamma'\subset\PP(L)\times \PP(V^*) \arrow[ld,"\pi_1'"'] \arrow[rd,"\pi_{n-1}'"] &  \\
\PP(L) \arrow[rr,"\restr{(\grad F)}{L}"',dashed] & & \PP(V^*)
\end{tikzcd}.
\end{equation*}
Then our numbers $\mu_i(L)$ are the projective degrees of $\restr{(\grad F)}{L}$, which are by definition the bidegrees of $\Gamma' \subset \PP(L) \times \PP(V^*)$ or degrees of the inverse image by $\restr{(\grad F)}{L}$ of a general projective subspace in $\PP(V^*)$ of a fixed dimension \cite[Definition 5.2]{cid2021mixed}. We stress that by the inverse image we mean the closure of all points in the domain \emph{on which the map is well defined} that map to the given set.

The final number $\mu_{\dim\hat{L}}(L)$ is equal to the product of the degree of $\grad(F)(L) \subset \PP(V^*)$ times the degree of the rational map $\restr{(\grad F)}{L}$. Indeed, let us fix a general subspace $S\subset \PP(V^*)$ of codimension equal to $\dim \hat L=\dim \grad(F)(L)$.
The number $\mu_{\dim\hat{L}}(L)$  is the cardinality of the set \[\{(x,y)\in \PP(L)\times \PP(V^*): \grad F(x)=y\text{ and }y\in S\}.\]
By definition of the degree we obtain $\deg \grad(F)(L)$ many points $y$ (which are general in $\grad(F)(L)$) and for each such $y$ we get $\deg \restr{(\grad F)}{L}$ many points $x$.

\emph{Assumption 2}: From now on we assume that the gradient is a generically one-to-one map on $\PP(L)$. In this case $\mu_{\dim\hat{L}}(L)$ is simply equal to the degree of $\grad(F)(L) \subset \PP(V^*)$.
This assumption will hold in all the examples we are interested in.

\subsection{Second construction: gradient of the restriction}

We can first restrict our map $F$ to $L$ and then take the gradient:
\begin{equation*}
\begin{tikzcd}
& \Gamma''\subset \PP(L)\times\PP(L^*) \arrow[ld,"\pi_1''"'] \arrow[rd,"\pi_{n-1}''"] &  \\
\PP(L) \arrow[rr,"\grad(\restr{F}{L})"',dashed] & & \PP(L^*)
\end{tikzcd}
\end{equation*}
\begin{definition} \label{def:nu}
We define $\nu_i(L)$ to be the projective degrees of $\grad(\restr{F}{L})$, which are by definition the bidegrees of $\Gamma'' \subset \PP(L) \times \PP(L^*)$.
The final number $\nu_{\dim \hat{L}}(L)$ is called the \emph{ML-degree} of $L$.
\end{definition}
Here ML stands for \emph{maximum likelihood}, a term very important in statistics. The connections to algebraic statistics will become apparent in Section~\ref{section:linearConcentrationsModels}. At this point we note that while $\Gamma'$ and $\hat L$ are naturally isomorphic, $\Gamma''$ is \emph{not}. Let us now explain the difference.

The upper triangle in the following diagram commutes by definition:
\begin{equation*}
\begin{tikzcd}[sep=huge]
 \PP(V) \arrow[rr,"\grad F",dashed]  & &\PP(V^*) \arrow[d,"\pi_{L^*}",dashed]    \\
\PP(L) \arrow[rr,"\grad(\restr{F}{L})"',dashed] \arrow[u,"\iota_L"] \arrow[rru,"\restr{\grad(F)}{L}",dashed] & & \PP(L^*)
\end{tikzcd}
\end{equation*}
Here $\pi_{L^*}$ is the projection from $L^{\perp}\subset V^*$ induced by the inclusion $L\subset V$.
\begin{lemma}
	We have: $\pi_{L^*} \circ \restr{(\grad F)}{L} = \grad(\restr{F}{L})$. In other words: the lower triangle in the above diagram also commutes.
\end{lemma}
\begin{proof}
	This can be easily seen using the algebraic definition of the gradient: since the map (\ref{eq:partiallySymmetricTensors}) is functorial in $V$, it follows that $\hook(\restr{F}{L})$ is given by the composition $\pi_{L^*} \circ (\hook F) \circ \iota_{S^{d-1} L}$. Combining this with the fact that $\nu_V \circ \iota_L = \iota_{S^{d-1}L} \circ \nu_L$ (where $\nu$ stands for the Veronese embedding) gives the desired result.
\end{proof}

The following proposition is a corollary of a very important result \cite[Proposition 2.1]{kohn2020linear}, with a proof proposed by Ranestad, implying that if we intersect a projective variety $X$ with a general subspace $T$ of codimension equal to dimension of $X$ containing a fixed subspace $S$, then the number of points (which are reduced by Bertini theorem) in $(T\cap X)\setminus S$  is at most the degree of $X$ with equality holding if and only if $S\cap X=\emptyset$.
\begin{prop} \label{prop:muIsNu}
	For every $i$, we have an inequality $\nu_i(L) \leq \mu_i(L)$. Further, equality holds for all $i$ if and only if $\nu_{\dim\hat{L}}=\mu_{\dim\hat{L}}$ if and only if $\overline{(\grad F)(L)} \cap L^{\perp} = 0$.
\end{prop}
\begin{proof}
Let $L'\subset L$ be a general subspace in $L$ of codimension $n-i$. By definition, $\mu_i(L)$ is the degree of the variety $\overline{(\grad F)(L')}$. On the other hand, $\nu_i(L)$ equals the number of pairs $(x,y)$ where $x\in \PP(L')$, $y\in \PP(L^*)$, $y=(\grad (\restr{F}{L}))(x)$ and $y$ belongs to a general subspace in $\PP(L^*)$ of codimension equal to $\dim \PP(L')$. We will identify the pair $(x,y)$ with the point $(\grad F)(x)\in \PP(V^*)$. Then $\nu_i(L)$ equals the number of points that:
\begin{itemize}
\item belong to the intersection of $\overline{(\grad F)(L')}$ with a general subspace of codimension $\dim \PP(L')$ that contains $\PP(L^\perp)$ and
\item do not belong to $\PP(L^\perp)$.
\end{itemize}

By \cite[Proposition 2.1]{kohn2020linear}, we see that $\nu_i(L)\leq \mu_i(L)$ and equality holds if and only if $\PP(L^\perp)\cap  \overline{(\grad F)(L')}=\emptyset$.

Hence, equality holds for all $i$ if and only if $\overline{(\grad F)(L)} \cap L^{\perp} = 0$, as $\overline{(\grad F)(L')}\subset \overline{(\grad F)(L)}$. The last equality holds if and only if $\nu_{\dim\hat{L}}=\mu_{\dim\hat{L}}$.
\end{proof}

\begin{remark}\label{aresultof}
	A result of Teissier \cite{T1}, \cite[Chapter 5]{T2} says that for $L$ a generic linear subspace, the condition $(\overline{\grad F)(L)} \cap L^{\perp} = 0$ in \Cref{prop:muIsNu} is always satisfied.
\end{remark}

As a first application, let us show how the numbers $\nu_i(L)$ allow us to compute the topological Euler characteristics of determinantal hypersurfaces, or more generally of $V(F)\cap\PP(L)$. These results were obtained in \cite[Section 3.1]{huh2012milnor} and \cite{dimca2003hypersurface}.
We first remark that if $V$ is the space of square matrices and $F$ is the determinant, then $V(F)\cap\PP(L)$ is the locus of degenerate matrices in $L$.
\begin{prop} \label{prop:EulerChar}
Let $X':=\PP(L)\setminus V(F)$ be the complement of the locus of points in $\PP(L)$ on which $F$ vanishes. Then the topological Euler characteristic:
\[\chi(X')=\sum_{i=0}^{\dim \PP(L)}(-1)^i\nu_i(L).\]
\end{prop}
\begin{proof}
Our numbers $\nu_i(L)$ are, by definition, the projective degrees of the gradient map of the restriction of $F$ to $\PP(L)$. By \cite[Remark 10]{huh2012milnor}, these are exactly the mixed multiplicities defined in \cite[Definition 8]{huh2012milnor}. The statement is then precisely the corollary after \cite[Theorem 9]{huh2012milnor}.
\end{proof}

\begin{example}
Consider the space $\PP(V)=\PP^2$ of symmetric $2\times 2$ matrices. The determinant $F$ defines the second Veronese embedding of $V(F)\simeq\PP^1$. Hence, $\chi(V(F))=2$. Further:
\[\chi(\PP^2\setminus V(F))=\chi(\PP^2)-\chi(V(F))=3-2=1.\]
Consider $L:=V$. In this case the gradient of $F$ is a linear, and hence everywhere defined, isomorphism. In particular:
\[\nu_0=\mu_0=\nu_1=\mu_1=\nu_2=\mu_2=1.\]
As in \Cref{prop:EulerChar} we have $\chi(\PP^2\setminus V(F))=1-1+1=1$.
\end{example}
\begin{example}
Consider the space $\PP(V)=\PP^3$ of general $2\times 2$ matrices. The determinant $F$ defines the Segre embedding of $V(F)\simeq\PP^1\times \PP^1$. Hence, $\chi(V(F))=2\cdot 2=4$. Further:
\[\chi(\PP^3\setminus V(F))=\chi(\PP^3)-\chi(V(F))=4-4=0.\]
Consider $L:=V$. In this case the gradient of $F$ is a linear, and hence everywhere defined, isomorphism. In particular:
\[\nu_0=\mu_0=\nu_1=\mu_1=\nu_2=\mu_2=\nu_3=\mu_3=1.\]
As in \Cref{prop:EulerChar} we have $\chi(\PP^3\setminus V(F))=1-1+1-1=0$.
\end{example}
In \Cref{eg:P3of3x3matrices}, we will consider a less trivial example, where the $\mu_i$ are not equal to the $\nu_i$.
\subsection{Linear concentration models} \label{section:linearConcentrationsModels}
To a symmetric positive-definite matrix $\Sigma$, one can associate a probability distribution, namely the Gaussian distribution with mean zero and covariance matrix $\Sigma$. Its probability density function is given by
\[f_{\Sigma}(x):=\frac{1}{\sqrt{(2\pi)^n\det\Sigma}}e^{-\frac{1}{2}x^T\Sigma^{-1}x}.\]

To a linear space $L \subset \Sym_2(\RR^n)$ one can associate a statistical model in two ways: the \emph{linear covariance model} is the set of all probability distributions $f_\Sigma$ where $\Sigma \in L$, and the \emph{linear concentration model} is the set of all probability distributions $f_\Sigma$ where $\Sigma^{-1} \in L$. We will here focus on the latter model. We can paramatrize it with the \emph{concentration matrix} $K := \Sigma^{-1}$, i.e.\ to a matrix $K \in L$ we associate the Gaussian distribution $f_{K^{-1}}$.

In maximum likehood estimation we are given a collection of samples $s_1, \ldots, s_M \in \RR^n$, and we want to find the parameter $K \in L$ that maximizes the likelihood function $\fL_{s}(K) := \prod_i f_{K^{-1}}(s_i)$, or equivalently the log likelihood function $\ell_{s}(K) := \log \fL_{s}(K)$. If we define the \emph{sample covariance matrix $S := \frac{1}{M}\sum_{i}{s_is_i^T}$}, then the log likelihood function can (up to additive and multiplicative constants) be rewritten as
\begin{equation}\label{eq:logLikelihood}
\ell_{S}(K) = \log(\det K) - \tr(S \cdot K).
\end{equation}
The \emph{maximum likelihood degree} of the linear concentration model is defined as the number of complex critical points of \eqref{eq:logLikelihood}, where the data $S$ is assumed to be generic. It is an algebraic measure of the complexity of maximum likelihood estimation for this particular model.

The critical points are precisely the matrices $K \in L$ for which
\begin{equation}\label{eq:criticalPoints}
\tr(K^{-1} \cdot X) = \tr(S\cdot X) \text{ for all } X \in L.
\end{equation}
 This can be seen by choosing a basis $K_1, \ldots, K_d$ of $L$, writing $K=\sum{\lambda_iK_i}$, and computing the partial derivatives $\frac{\partial \ell_{S}(K)}{\partial \lambda_i}$ using Jacobi's formula.
Condition \eqref{eq:criticalPoints} can be rewritten as: $K^{-1}-S \in L^{\perp}$.

Thus, in order to compute the maximum likelihood degree, we have to count the number of intersection points of $L^{-1}$ with a generic affine translate of $L^{\perp}$.
On the other hand, from the proof of \Cref{prop:muIsNu}, it follows that $\nu_{\dim\hat{L}}(L)$ is equal to the number of points in $\PP(L^{-1}) \cap \PP(L^{\perp} + R) \setminus \PP(L^{\perp})$, where $R$ is a generic $1$-dimensional subspace. This number is the same as our maximum likelihood degree, because for any cone $C$, linear space $L$, and point $p \notin L$, there is a bijection between $C \cap (L+p)$ and $\PP(C) \cap \PP(L + \CC\cdot p) \setminus \PP(L)$.
Summarizing, we have shown the following result, which justifies the name ``ML-degree" for $\nu_{\dim\hat{L}}(L)$:
\begin{prop}
	If $L \subset V := S^2\CC^n$ is a linear space of symmetric matrices and $F$ is the determinant (so that the map $\grad : \PP(V) \to \PP(V^*)$), the ML-degree $\nu_{\dim\hat{L}}(L)$ from \Cref{def:nu} equals the maximum-likelihood degree of the linear concentration model associated to $L$.
\end{prop}
\begin{remark}
	In the recent collaboration project ``Linear Spaces of Symmetric Matrices" \cite{LSSM}, a lot of progress has been made in computing the numbers $\mu(L) := \mu_{\dim\hat{L}}(L) = \deg L^{-1}$ and $\nu(L) := \nu_{\dim\hat{L}}(L) = \operatorname{ML-degree}(L)$ for specific linear spaces of symmetric matrices. Here is an overview:
	\begin{itemize}
		\item For $\dim L=2$, the possible $L$ can by classified by \emph{Segre symbols}, and both $\mu(L)$ and $\nu(L)$ can be read off from the Segre symbol \cite{FMS}.
		\item For $n=3$ and $\dim L=3$, there is a classification of all possible $L$ due to C.T.C.\ Wall \cite{Wall}. In each case, $\mu(L)$ and $\nu(L)$ have been computed in \cite{DKRS}.
		\item Spaces $L$ for which $\mu(L)=1$ have been classified in \cite{Jordan}. They are \emph{Jordan algebras}.
		\item The \emph{reciprocal ML-degree of the Brownian motion tree model}, for which an explicit formula was obtained in \cite{Brownian}, is a special case of the ML-degree $\nu(L)$.
		\item For $L$ equal to the \emph{space of catalecticant matrices associated with ternary quartics}, $\mu(L)$ and $\nu(L)$ have been computed numerically in \cite{BCH}.
		\item For the linear space of symmetric matrices associated to the $n$-cycle, a formula for $\mu(L)$ has been proven in \cite{cycle}.
		\item For $L$ a \emph{generic} linear space of symmetric matrices, a formula has been obtained in \cite{MMMSV}, see also \Cref{sec:polynomiality}.
	\end{itemize}
The last three results rely on using intersection theory on the space of \emph{complete quadrics}. We will introduce this variety in \Cref{sec:CQ}, and discuss its cohomology in \Cref{completequad}.
\end{remark}

\begin{remark} \label{rmk:manivel1}
	We have just shown that for a linear \emph{concentration} model, the ML-degree is the number of invertible matrices $K \in L$ for which $K^{-1} - S \in L^{\perp}$, where $S$ is a generic symmetric matrix. For linear \emph{covariance} models, there is a similar description: the ML-degree is the number of invertible matrices $\Sigma \in L$, for which $\Sigma^{-1}S\Sigma^{-1} - \Sigma^{-1} \in L^{\perp}$ \cite[Proposition 3.1]{STZ}. Making a connection to intersection theory is significantly harder in this case. Nevertheless this ML-degree, sometimes referred to as the \emph{reciprocal} ML-degree, is understood if $L$ is 2-dimensional~\cite{FMS}, if $L$ is a 3-dimensional space of $3 \times 3$-matrices \cite{DKRS}, if $L$ consists of diagonal matrices~\cite{EFSS}, or if $L$ is generic and at most $4$-dimensional~\cite{manivel2021proof}. For that last case, see also \Cref{rmk:manivel2}.
\end{remark}

\section{ML-degrees and Segre classes}

The $\mu_i(L)$ and $\nu_i(L)$ may be computed using the Segre classes. These constructions are based on \cite{FultonIntersection}, in particular Chapter 9. The reader may consult also a very nice, new presentation of Segre classes and their applications  \cite{aluffisegre}. Our general setting is that we are given a variety $X\subset\PP^n$ and a closed subscheme $Y\subset X$. Any reader should not be afraid here of the word subscheme, it simply means that we take any homogeneous ideal that contains the ideal $I_X$ of $X$. The presented constructions are quite technical. Readers unfamiliar with characteristic classes of vector bundles may skip this section, remembering that Segre classes are currently objects that may be computed effectively e.g.~in Macaulay2~\cite{M2} and handled well from the theoretical point of view.

\subsection{Segre classes of cones}

We introduce Segre classes in a more general setting, namely Segre classes of cones. This will cover both the Segre classes of vector bundles and Segre classes of subschemes.
\begin{definition}[{\cite[Section 4.1]{FultonIntersection}}] Let $C$ be a cone over $X$ i.e. $C=\Spec(S^{\bullet})$ where $S^{\bullet}$ is a sheaf of graded $\mathcal{O}_X$-algebras. We assume that $\mathcal{O}_X \rightarrow S^{0}$ is surjective, $S^{1}$ is coherent and $S^{\bullet}$ is generated by $S^{1}$. Let $\PP(C\oplus 1)=\Proj(S^{\bullet}[z])$ be the projective completion of $C$, the projection $\eta: \PP(C\oplus 1) \rightarrow X$ and $\mathcal{O}(1)$ the canonical line bundle on $\PP(C\oplus 1)$. The Segre class of $C$ is defined as
\[
s(C)=\sum_{i\geq 0} \eta_{*}(c_1(\mathcal{O}(1))^i \cap [\PP(C\oplus 1)]).
\]
\end{definition}

\begin{remark}
Let $\mathcal{E}$ be a vector bundle of rank $r$. We associate to it the sheaf of $\mathcal{O}_X$-algebras $\Sym(\mathcal{E}^{\vee})$. In this special case it turns out that we can work with the projectivization $\PP(C)=\PP(\mathcal{E})$ instead of the projective completion $\PP(C\oplus 1)=\PP(\mathcal{E}\oplus \mathcal{O}_X)$. This leads to
\[
s_i(\mathcal{E})= \pi_{*}(c_1(\mathcal{O}_{\PP(\mathcal E)}(1))^{r-1+i}),
\]
where $\pi : \PP(\mathcal{E})\rightarrow X$.
\end{remark}

\begin{example}
Let $X=G(2,4)$ be the Grassmannian of 2-planes in $\CC^4$ and let $\mathcal{E}$ be the tautological bundle. Then
$E=\{(L,p): p\in L\} \subset G(2,4)\times \CC^4$ is the total space of $\mathcal{E}$ and $\PP(E)=\{(L,[p]): p\in L\} \subset G(2,4)\times \PP(\CC^4)$ is the total space of $\PP(\mathcal E)$. The canonical bundle $\mathcal{O}_{\PP(\mathcal E)}(1)$ has as fiber over $(L, [p])$ the dual of $\CC p \subset \CC^4$. So the global sections are given by linear functions on $\CC^4$. Since the first Chern class is the vanishing locus of a global section, we get
\[
\xi:=c_1(\mathcal{O}_{\PP(\mathcal E)}(1))=[\{(L, [p]): p\in H\}] \in A^1(\PP(E)),
\]
where $H\subset \CC^4$ is a hyperplane. Thus $$s_0(\mathcal{E})=\pi_{*}(\xi)=[\{L \in G(2,4): L\cap H \neq 0\}]=[G(2,4)].$$ We have $\xi^2=[\{(L, [p]) : p \in V_2\}]$ where $V_2 \subset \CC^4$ is a 2-plane, thus $$s_1(\mathcal E)= [\{L \in G(2,4): L\cap V_2 \neq 0\}]\in A^1(G(2,4)).$$
\end{example}

We will be interested later in Segre classes of subschemes. By definition, the Segre class $s(Y, X)$ of a subscheme $Y\subset X$ is the Segre class of its normal cone.
Let $\pi: \tilde{X} \rightarrow X$ be the projection of the blow-up of $X$ along $Y$. We denote the exceptional divisor by $\tilde{Y}=\pi^{-1}(Y)$ and $\eta=\restr{\pi}{\tilde{Y}}: \tilde{Y}\rightarrow Y$.
\begin{remark}\cite[Corollary 4.2.2]{Fulton}
Alternatively, we can compute the total Segre class of a subscheme as
\[
s(Y,X)=\sum_{k\geq 1}(-1)^{k-1}\eta_{*}\left(\restr{[\tilde{Y}]^k}{\tilde{Y}}\right).
\]
\end{remark}

\subsection{Domain approach}

We start with the case $X=\PP^n$. Fix $f_0,\dots,f_m$ to be homogeneous polynomials of degree $d$ in $n+1$ variables. We let $Y$ be defined by the ideal $I_Y:=\langle f_0,\dots,f_m\rangle$. In the setting introduced above, we are interested in the multidegree of the graph of the rational map $\PP^n\dashrightarrow \PP^m$ defined by $f:=[f_0:\dots:f_m]$. As the next theorem shows, computing these numbers is essentially equivalent to computing the degrees $\deg s_i$ of the Segre class $s(Y,\PP^n):=\sum_{i=0}^{\dim Y} s_i$ for $s_i\in A_{\dim Y-i}(Y)$. Here we consider $s_i$ as a linear combination of (classes of) subvarieties of $Y$, hence also of $\PP^n$. The degree $\deg s_i$ is the same linear combination of degrees of those projective subvarieties.
\begin{theorem}
Using the notation above, fix a general $\PP^a\subset \PP^m$ and assume that
$\restr{f}{f^{-1}(\PP^a)}$
is generically bijective. Let $V:=f^{-1}(\PP^a)\setminus Y$
Then:
\[d^{m-a}=\deg V+\sum_{i=0}^{m-a-n+\dim Y}\binom{m-a}{n-\dim Y+i}d^{m-a-n+\dim Y-i}\deg s_i.\]
\end{theorem}
\begin{proof}
This follows directly from \cite[Theorem 3.2]{Eklund}. We simply replaced the choice of general polynomials in the ideal (of fixed degree), by a choice of general $\PP^a$.
\end{proof}
Informally, the theorem may be interpreted as follows. When we pick $m-a$ polynomials of degree $d$, by B\'ezout we expect that the degree of the zero locus equals $d^{m-a}$. However, if we pick general polynomials from some linear subspace of degree $d$ polynomials, there is a correction term coming from the Segre classes of the base locus of that subspace.
\begin{corollary}\label{cor:ChowDomain}
Let $F$ be a homogeneous polynomial on $\PP(V)$ and $\PP(L) \subset \PP(V)$ be a subspace of projective dimension $n$. Let $Y$ be the intersection of $\PP(L)$ with the (scheme) $\grad F=0$. Let $m:=\dim Y$.  We have:
\[\mu_i(L)=(\deg F-1)^{i}-\sum_{j=0}^{i-n+m}\binom{i}{n-m+j}(\deg F-1)^{i-n+m-j}\deg s_j,\]
where $s_j$ is the $j$-th component of $s\left(Y,\PP(L)\right)$.%

Similarly:
\[\nu_i(L)=(\deg F-1)^{i}-\sum_{j=0}^{i-n+m}\binom{i}{n-m+j}(\deg F-1)^{i-n+m-j}\deg s'_j,\]
where $s'_j$ is the $j$-th component of $s(\grad (F_{|\PP(L)})=0,\PP(L))$.
\end{corollary}

The above corollary underlines the main difference between $\nu_i$ and $\mu_i$ from a geometric perspective. Indeed, given a polynomial $F$ on $\PP(V)$ its singular locus is defined by the ideal $\grad F$. But how to describe the singular locus, when we intersect with $\PP(L)$?
In principle, there are two possible answers.

One, is that we scheme-theoretically intersect $\grad F=0$ with $\PP(L)$.  The other one, is to first restrict $F$ to $\PP(L)$ and then take the gradient. Both constructions give us ideals, which do not have to be equal, even if $\PP(L)$ is general!

Indeed, Bertini theorem tells us that the radicals of both ideals are equal\footnote{Assuming of course that $\PP(L)$ is general.}, as the varieties associated to both ideals are just equal to (the singular locus of $V(F)$) intersected with $\PP(L)$, which is the singular locus of $V(F)\cap \PP(L)$.

Teissier \cite{T1,T2} tells us that as long as $\PP(L)$ is general, the two ideals have the same integral closure. This implies that $\nu_i=\mu_i$, although the schemes/ideals may be different.

For special $\PP(L)$, as we will see, the sequences $\nu_i$ and $\mu_i$ are distinct.

\begin{example}
Consider $\PP(V)=\PP^3$ with coordinates $[a:b:c:d]$. Let $\PP(L)\subset\PP(V)$ be given by $a+b+c+d=0$. Let $F(a,b,c,d)=abcd$. The rational map $\PP(L)\dashrightarrow \PP^3$ given by $\grad F$ has multidegree $\mu_0=1,\mu_1=3,\mu_2=3$.

This may be confirmed with the following Macaulay2 code \cite{M2}:
\begin{verbatim}
loadPackage"Cremona"
R=QQ[a,b,c,d]
S=QQ[x,y,z,t]
L=R/ideal(a+b+c+d)
phi=map(L,S,{a*b*c,a*b*d,a*c*d,b*c*d})
projectiveDegrees(phi,MathMode=>true)
\end{verbatim}
The base locus of the gradient map is given by the ideal
\begin{verbatim}
Y=ideal(a*b*c,a*b*d,a*c*d,b*c*d);
\end{verbatim}

We may compute the Segre class $s(Y,\PP(L))$, pushing it forward to $\PP(V)$:
\begin{verbatim}
SegreClass(Y,MathMode=>true)
\end{verbatim}
The output: $6H^3$ tells us that $s_0=0, s_1=6$. We may now reconfirm the computations by the first formula in Corollary \ref{cor:ChowDomain}. For example:
\[\mu_2=3^{2}-\binom{2}{1}\cdot 3\cdot 0-\binom{2}{2}\cdot 6=9-6=3.\]

We may also restrict the polynomial $F$ to $\PP(L)$ and take the gradient obtaining an ideal $Y'$.
\begin{verbatim}
F=a*b*c*d
Y'=ideal(diff(b, F), diff(c, F),diff(d, F))

\end{verbatim}
The ideals $Y'$ and $Y$ are not equal, however the integral closure of $Y'$ equals $Y$. The Segre class of $Y$ and $Y'$ in $\PP(L)$ is the same, hence $\nu_i=\mu_i$ for all $i$. As we will see this always holds (for any $\PP(L)$) for a polynomial that is a product of variables.

Surprisingly (but not coincidentally) the multidegree computes the coefficients of the (reduced) chromatic polynomial of a graph, that is a four-cycle. The general theorem will be provided in Section \ref{subs:permutvar}.
\end{example}
\begin{example}\label{exm:different}
Consider $\PP(V)=\PP^9$ be the space of $4\times 4$ symmetric matrices. Let $\PP(L)\subset\PP(V)$ be given by
\[\begin{pmatrix}
a & b & 0 & h\\
b & c & d & 0\\
0 & d & e & f\\
h & 0 & f & g
\end{pmatrix}
\]
Let $F$ be the determinant. The rational map $\PP(L)\dashrightarrow \PP^9$ given by $\grad F$ has multidegree $\mu=\{1,3,9,17,21,21,17,9\}
$.
\begin{verbatim}
R=QQ[a,b,c,d,e,f,g,h,x,y]
M=matrix{{a,b,x,h},{b,c,d,y},{x,d,e,f},{h,y,f,g}}
F=det M
ff=diff(vars R,F)
L=R/ideal(x,y)
S=QQ[a1,b1,c1,d1,e1,f1,g1,h1,x1,y1]
phi=map(L,S, first entries ff)
loadPackage"Cremona"
projectiveDegrees(phi,MathMode=>true)
\end{verbatim}

However, if we first restrict $F$ to $L$, the multidegree of the graph of the gradient changes to $\nu=\{1,3,9,17,21,21,15,5\}$.

\begin{verbatim}
S2=QQ[a1,b1,c1,d1,e1,f1,g1,h1]
proj=map(S,S2,sub(vars S2,S))
phi2=phi*proj
projectiveDegrees(phi2,MathMode=>true)
\end{verbatim}

One can provide intuition behind the difference between the last numbers in $\mu$ and $\nu$, i.e.~$9-5=2\cdot2$. Indeed, $9$ is the degree
of the variety $L^{-1}$ obtained from $L$ by taking the gradient of the determinant. The number $5$ is the degree of the projection of that variety $L^{-1}$ from $L^\perp$. If $L^{-1}$ and $L^\perp$ are disjoint (in the projective space) then the two numbers coincide. This happens when $L$ is generic and always when $L$ is (simultanously) diagonalizable. However, in this example $L^{-1}\cap L^\perp$ consists of two (reduced) points in $\PP^9$. Moreover, they both contribute with multiplicity 2 in the computation of $\mu_7$ since they are singular points of $L^{-1}$.
\begin{remark}
A more general way of understanding the difference between $\mu$ and $\nu$ will be presented in the next section (\Cref{thm:muminusnu}). In the present case where $L^{-1} \cap L^{\perp}$ is zero-dimensional, one would hope for the much easier formula $$\nu_{\dim \hat{L}}=\mu_{\dim \hat{L}}-\deg(L^{-1}\cap L^{\perp}).$$ This is true if $L^{-1}\cap L^{\perp}$ only consists of smooth points of $L^{-1}$, see \cite[Corollary 4.3]{AmendolaEtAl}, but in the current example it fails ($5 \neq 9-2$).
\end{remark}

The number $9$ is the result we obtain for $m=4$ in \cite[Conjecture 2]{sturmfels2010multivariate}, recently proved in \cite{cycle}. The number $5$ is the result we obtain for $m=4$ in the formula conjectured in \cite[Section 7.4]{drton2008lectures}, which still remains open.
\end{example}

\begin{remark}
There are many implemented methods to compute the Segre classes, e.g.~\cite{Eklund, harris2017computing, aluffi2003computing, helmer2015algorithms, harris2020segre}. In principle, any of those may be used to obtain the numbers $\mu_i$ and $\nu_i$ that we are interested in. It seems fair to point out that currently one of the best methods to compute the Segre classes goes the other way around: first compute the analogues of $\mu_i$'s (e.g.~numerically) and deduce the Segre class from this data.
\end{remark}

\subsection{Codomain approach}\label{codomain}
As we have seen in Example \ref{exm:different} the difference between $\mu$ and $\nu$ comes essentially from the intersection $L^{-1}\cap L^\perp$. This intuition was precisely formalized in \cite{AmendolaEtAl}. The authors only state the result for subspaces of symmetric matrices, however the statement holds, using exactly the same proof, in the following generality. Let $\PP^a\simeq \PP(L)\subset \PP^n$. As always we assume that $\PP(L)$ is not contained in the base locus of the gradient of a homogeneous polynomial $F$. We also assume that $\grad F$ is generically one-to-one on $\PP(L)$ and so is the restriction of the projection $\pi :\PP(V^*)\dashrightarrow \PP(L^*)$ restricted to $(\grad F)(\PP(L))$. Let $b:=\dim \PP(L^\perp)\cap \overline{(\grad F)(\PP(L))}$.

Note that $\PP(L^\perp), \overline{\grad F(\PP(L))}\subset \PP^n\hookrightarrow \PP^n\times \PP^n$, where the last inclusion is the diagonal inclusion. In particular, we may regard $\PP(L^\perp)\cap \overline{\grad F(\PP(L)}$ as a subvariety of $\PP^n$ and of $\PP(L^\perp)\times \overline{\grad F(\PP(L))}$. In particular, it makes sense to consider the Segre class $s(\PP(L^\perp)\cap \overline{\grad F(\PP(L))}, \PP(L^\perp)\times \overline{\grad F(\PP(L))})$ and to compute its degree as a class in the Chow ring of $\PP^n$. Note that $s(\PP(L^\perp)\cap \overline{\grad F(\PP(L))}, \PP(L^\perp)\times \overline{\grad F(\PP(L))})$ decomposes as a sum of classes, according to the grading of the Chow ring, and each degree $i$ component, has its degree as a (linear combination of) subscheme(s) in $\PP^n$. We will denote the degree (as a subscheme) of the degree $i$-part (in the Chow ring) simply by $s_i(\PP(L^\perp)\cap \overline{\grad F(\PP(L))}, \PP(L^\perp)\times \overline{\grad F(\PP(L))})$.
\begin{theorem}[{See \cite[Theorem 4.2]{AmendolaEtAl}}] \label{thm:muminusnu}
Using the notation above we have the following.
The ML-degree of $L$ equals:
\[\nu_a(L)=\mu_a(L)-\sum_{j=0}^b \binom{n}{j} s_j \left(\PP(L^\perp)\cap \overline{\grad F(\PP(L))}, \PP(L^\perp)\times \overline{\grad F(\PP(L))}\right). \]
\end{theorem}
\begin{example}[Based on Example 4.5 in \cite{AmendolaEtAl}] \label{eg:P3of3x3matrices}
Consider the space $L$ to be the following subspace of the symmetric $3\times 3$ matrices:
\[
\begin{pmatrix}
0 & a & b\\
a & 0 & c\\
b & c & d
\end{pmatrix}.
\]
As the following code shows we have $\mu=(1, 2, 4, 4)$.
\begin{verbatim}
R=QQ[a,b,c,d,x,y]
M=matrix{{x,a,b},{a,y,c},{b,c,d}}
F=det M
ff=diff(vars R,F)
L=R/ideal(x,y)
S=QQ[a1,b1,c1,d1,x1,y1]
phi=map(L,S, first entries ff)
loadPackage"Cremona"
projectiveDegrees(phi,MathMode=>true)
\end{verbatim}
Here $\PP(L^\perp)\cap\PP(L^{-1})$ is one-dimensional, i.e.~there will be two graded pieces of the Segre class contributing to the difference among $\mu$ and $\nu$. The degrees of these Segre classes are $-7$ and $2$, thus we get:
\[\nu_3=4+7-5\cdot 2=1.\]

Indeed, we have $\nu=(1, 2, 2, 1)$ as the following code shows:
\begin{verbatim}
S2=QQ[a1,b1,c1,d1]
proj=map(S,S2,sub(vars S2,S))
phi2=phi*proj

\end{verbatim}

The alternating sum of the $\nu_i$ is equal to $0$.
This is consistent with \Cref{prop:EulerChar}:
\begin{align*}
\chi\left(\PP(L) \setminus V(\restr{\det}{L})\right) &= \chi(\PP^3) - \chi \Big(V(a(2bc-ad))\Big)\\
&= 4-\chi \left(V(a)\right)- \chi \left(V(2bc-ad)\right) + \chi \left(V(a,bc)\right) \\
&=4-3-4+3 = 0.
\end{align*}

We may also confirm the formulas in Corollary \ref{cor:ChowDomain}.
The base locus $Y$ in this case is zero-dimensional, but not reduced:
\begin{verbatim}
Y=sub(ideal(ff),L)
\end{verbatim}
Hence, the only Segree class we have is $s_0$ of degree four (which in this case could also be computed by the standard degree computation):
\begin{verbatim}
SegreClass(Y)
\end{verbatim}
Hence, $\mu_3=2^3-\binom{3}{3}2^0\cdot 4=4$.

If we restrict the determinant to $\PP(L)$ the singular locus is one dimensional. Hence, we obtain two Segre classes $s_0$ and $s_1$ of degrees $2$ and $-5$ respectively:
\begin{verbatim}
Y'=ideal(diff(a,F),diff(b, F), diff(c, F),diff(d, F))
SegreClass(Y')
\end{verbatim}
Thus, for example, we have:
\[\nu_3=2^3-\binom{3}{2}2^1\cdot 2-\binom{3}{3}2^0\cdot(-5)=1.\]
\end{example}
\section{Complete varieties: examples}

Our numbers $\mu_i(L)$ can be computed as follows:
\[
\mu_i(L) = \int_{\hat{L}}{H_1^{d-i} \cdot H_2^{i}} = \int_{\Gamma}{H_1^{d-i} \cdot H_2^{i} \cdot [\hat{L}]}.
\]

In the diagram (\ref{eq:diagramGradF}), the graph $\Gamma=\Gamma_{\grad F}$ is usually not smooth. We would like to replace it by a smooth variety, i.e.\ find a birational morphism $\tilde{\Gamma} \to \Gamma$ where $\tilde{\Gamma}$ is smooth.

\begin{definition}
	A \emph{complete variety} with respect to a homogeneous polynomial $F \in S^d V^*$ is a smooth variety $\tilde{\Gamma}$ equipped with two maps $\pi: \tilde{\Gamma} \to \PP(V)$ and $\pi': \tilde{\Gamma} \to \PP(V^*)$ such that following diagram commutes:
	\begin{equation}\label{eq:diagramCompleteVariery}
	\begin{tikzcd}
	& \tilde{\Gamma} \arrow[ld,"\pi"'] \arrow[rd,"\pi'"] &  \\
	\PP(V) \arrow[rr,"\grad F"',dashed] & & \PP(V^*)
	\end{tikzcd}
	\end{equation}
\end{definition}

Our aim is to have a simple intersection theory on $\tilde{\Gamma}$, in particular, we want to construct it effectively.
Suppose we have such a $\tilde{\Gamma}$. For $\PP(L) \subset \PP(V)$, we let $\PP(\tilde{L})$ denote the strict transform of $L$ over $\tilde{\Gamma} \to \PP(V)$. Then
\begin{equation}\label{eq:muChow}
\mu_i(L) = \int_{\tilde{\Gamma}}{H_1^{d-i} \cdot H_2^{i} \cdot [\tilde{L}]}.
\end{equation}
So if we can describe the Chow ring of $\tilde{\Gamma}$, and compute the class $[\tilde{L}] \in A_d(\tilde{\Gamma})$, then we can use the above formula to compute $\mu_i(L)$. In the rest of this section we will construct complete varieties $\tilde{\Gamma}_F$ for some specific choices of $F$. In the next section we will describe their cohomology ring.

\subsection{The permutohedral variety}\label{subs:permutvar}
Let $\PP(V) = \PP^n$; we fix a basis $e_0, \ldots, e_n$ of $V$. Let $F$ be the polynomial $x_0x_1\cdots x_n$. Then $\grad F: \PP(V) \dashrightarrow \PP(V^*)$ is the Cremona map $[x_0:x_1:\ldots:x_n] \mapsto [x_0^{-1}:x_1^{-1}:\ldots:x_n^{-1}]$. We will write $L^{-1} := \overline{(\grad F)(L)}$.

\begin{lemma}
	$L^{-1} \cap L^{\perp} = 0$; hence $\mu_i(L)=\nu_i(L)$ for all $i$.
\end{lemma}
\begin{proof}
	Suppose we have $y=(y_0,\ldots,y_n) \in L^{-1}$; write $I:=\{i \mid y_i \neq 0\} \subset [n+1]$. There is a sequence $(x_j)_{j \in \NN}$ with $x_j = (x_{j,0}, \ldots, x_{j,n}) \in L \cap (\CC^*)^{n+1}$ and $\lim_{j \to \infty}{x_j^{-1}}=y$, i.e.\ \[\lim_{j \to \infty}{x_{j,i}}=\begin{cases}
	y_i^{-1} & \text{for } i \in I,\\
	\infty & \text{for } i \notin I.
	\end{cases}\]
	So we find that
	\[\lim_{j \to \infty}{\sum_{i=0}^{n}{y_ix_{j,i}}} = \lim_{j \to \infty}{\sum_{i \in I}{y_ix_{j,i}}}=|I|.\]
	If we had $y \in L^{\perp}$, we would have $\sum_{i=0}^{n}{y_ix_{j,i}} = 0$ for all $j$, contradicting the above.
\end{proof}

There is a beautiful connection between our numbers $\nu_i(L)$ and matroid theory.
A linear subspace $L \subset V$ gives rise to a matroid $M(L)$ on $[n+1] = \{0,1,\ldots,n\}$: a subset $S \subseteq [n+1]$ is independent if $\{e_i \mid i \in S\}$ is linearly independent modulo $L$.
One of the most important matroid invariants is the characteristic polynomial, which generalizes the chromatic polynomial of a graph. The definition can be found for instance in \cite[Section 2]{HuhKatz}.

\begin{theorem}[{\cite[Theorem 1.1]{HuhKatz}}] \label{thm:HuhKatz}
The numbers $\nu_i(L)$ are the coefficients of the reduced characteristic polynomial of $M(L)$.
\end{theorem}
This geometric interpretation of the characteristic polynomial was the key ingredient in the proof of the longstanding Rota-Heron-Welsh conjecture\cite{HuhKatz,AdiprasitoHuhKatz}, which states that its coefficients form a log-concave sequence.

In our current setting, a desingularization of $\Gamma$ is given by the \emph{permutohedral variety} $\Per(V)=\Per_n$. This variety can be described as the closure of the image of the rational map
\[
\Delta F: \PP(V) \to \PP(V) \times \PP(\bigwedge^2 V) \times \cdots \times \PP(\bigwedge^{n-1} V):
\]
\[
[x] \mapsto ([x_0:\ldots:x_n],[x_0x_1:x_0x_2:\ldots:x_{n-1}x_n],\ldots,[x_0\cdots x_{n-1}:\ldots:x_1\cdots x_{n}]).
\]
Alternatively, $\Per_n$ can be constructed as a repeated blow-up. Consider the sequence
\begin{equation} \label{eq:blowupPer}
\PP^n = X_0 \longleftarrow X_1 \longleftarrow \cdots \longleftarrow X_{n-1},
\end{equation}
where $X_1$ is obtained from $X_0$ by blowing up the points $\{[e_i]\mid 0 \leq i \leq n\}$; $X_2$ is obtained from $X_1$ by blowing up the strict transform of the lines $\{\langle e_i,e_j \rangle \mid 0 \leq i < j \leq n\}$, and in general $X_{k+1}$ is obtained from $X_{k}$ by blowing up the strict transform of the (projective) $k$-planes $\{\langle e_{i_0},\ldots e_{i_k} \rangle \mid 0 \leq i_0 < \ldots < i_k \leq n\}$. Then the variety $X_{n-1}$ is isomorphic to the permutohedral variety $\Per_n$.

The first construction clearly resolves the graph of $\grad F$, via the projection maps $\pi_1: P_n \to \PP(V)$ and $\pi_{n-1}: P_n \to \PP(\bigwedge^{n-1}V) \cong \PP(V^*)$. The second is smooth, as can be seen from the fact that blowing up a smooth subvariety of a smooth variety again yields a smooth variety. To see that both constructions agree, one can inductively show that the variety $X_k$ obtained after $k$ blow-ups is isomorphic to the image closure of the map $\PP(V) \to \PP(V) \times \PP(\bigwedge^2 V) \times \cdots \times \PP(\bigwedge^{k+1} V)$. Putting all of this together we see that indeed, $\Per_n$ is a complete variety.

A third construction of $\Per_n$ goes via toric geometry. This we will explained in detail in Section \ref{sec:perCohomology}.

\subsection{The space of complete collineations}
We take $V=W \ot W$, where $W$ is a $\CC$-vector space of dimension $d$. We think of $V$ as the space of $d\times d$ matrices, and take $F := \det \in S^d V^*$. Then $\grad F: \PP(V) \dashrightarrow \PP(V^*)$ is the projectivization of the map $V \to V^*: A \mapsto \adj(A)$ mapping a square matrix to its adjugate (which equals $\det(A)\cdot A^{-1}$ if $A$ is invertible).

For $A \in W \ot W$, we write $\bigwedge^kA \in \bigwedge^k W \ot \bigwedge^k W$ for the $k$-th \emph{compound matrix}. In coordinates, $\bigwedge^kA$ is the $\binom{d}{k} \times \binom{d}{k}$-matrix whose entries are the $k \times k$-minors of $A$. The coordinate-free description is as follows: if we view $A$ as a linear map $W^* \to W$, then $\bigwedge^kA: \bigwedge^k(W)^* \cong  \bigwedge^k(W^*) \to \bigwedge^kW$. Note that, up to signs, $\bigwedge^{d-1}{A}=\adj(A)$, where $\adj(A)$ is the adjugate matrix of $A$. %

\begin{definition} \label{def:CC}
	The variety $\ComplCol(W)$ of \emph{complete collineations} is the closure of the image of the set of invertible matrices under the map
	\begin{align*}
	\varphi: \PP(W \ot W) \dashrightarrow {\PP\left(W \ot W\right)} \times {\PP\left(\bigwedge^2W \ot \bigwedge^2W\right)} \times \cdots \times {\PP\left(\bigwedge^{d-1}W \ot \bigwedge^{d-1}W\right)},
	\end{align*}
	sending a matrix $A$ to $(A,\bigwedge^2 A, \ldots, \bigwedge^{d-1} A)$.
	This is a complete variety with respect to $F := \det \in S^d V^*$.
\end{definition}

Here is an alternative construction of the space of complete collineations:
Start with the space $\PP(W \ot W)$ and consider the following sequence of blow-ups
\begin{equation} \label{eq:blowup}
\PP(W \ot W) = X_0 \longleftarrow X_1 \longleftarrow \cdots \longleftarrow X_{d-2} =: \ComplCol(W),
\end{equation}
where $X_i$ is the blow-up of $X_{i-1}$ at the strict transform of the locus of rank $i$ matrices.

\begin{remark}
	If we intersect $\ComplCol(W)$ with the space of tuples of diagonal matrices, we obtain the permutohedral variety $\Per_{d-1}$.
\end{remark}

Much more on complete collineations can be found for instance in \cite{ThaddeusComplete,Massarenti1,DeConciniProcesi1}.

{\bf Degeneration spaces.}
For each $r \in \{1,\ldots,d-1\}$, we define a subvariety $S_r \subset \ComplCol(V)$:
\begin{align} \label{eq:Sr}
S_r =& \{(A_1, \ldots, A_{d-1}) \in \ComplCol(W) \mid \rk A_r = 1\} \nonumber \\
=& \overline{\{(A_1, \ldots, A_{d-1}) \in \ComplCol(W) \mid \rk A_1 = r\}} \\
=& \overline{\{(A_1, \ldots, A_{d-1}) \in \ComplCol(W) \mid \rk A_{d-1} = d-r\}}. \nonumber
\end{align}
Alternatively, $S_r$ is the strict transform of the exceptional locus of the blow-up $X_{r-1} \longleftarrow X_r$ in (\ref{eq:blowup}). From this last description it follows that $S_r \subset \ComplCol(W)$ is irreducible of codimension one. A general point in $S_r$ is determined by a rank $r$ matrix $A_1$ and rank $d-r$ matrix $A_{d-1}$ that vanishes on the image of $A_1$. The image of $A_1$ (resp.~kernel of $A_{d-1}$) is parameterized by the Grassmannian $G(r,d)$. Hence, there is a natural map from $S_r$ to $G(r,d)$. The cohomology class of $S_r$ will play a central role in Section~\ref{completequad}.

\subsection{The space of complete quadrics} \label{sec:CQ}

We take $V=S^2W$ the space of symmetric $d \times d$ matrices, and $F=\det \in S^dV^*$ the determinant of a general symmetric matrix. Analogously to the previous section, we obtain:

\begin{definition} \label{def:CQ}
	The variety $\CQ(W)$ of \emph{complete quadrics} is the closure of the image of the set of invertible matrices under the map
	\begin{align*}
	\varphi: \PP(S^2W) \dashrightarrow {\PP\left(S^2W\right)} \times {\PP\left(S^2(\bigwedge^2W)\right)} \times \cdots \times {\PP\left(S^2(\bigwedge^{d-1}W)\right)},
	\end{align*}
	sending a matrix $A$ to $(A,\bigwedge^2 A, \ldots, \bigwedge^{d-1} A)$. This is a complete variety with respect to $F := \det \in S^d V^*$.
\end{definition}
Alternatively, the space of complete quadrics can be constructed via the following sequence of blow-ups:
\begin{equation} \label{eq:blowupCQ}
\PP(S^2 W) = X_0 \longleftarrow X_1 \longleftarrow \cdots \longleftarrow X_{d-2} =: \CQ(W),
\end{equation}
where $X_i$ is the blow-up of $X_{i-1}$ at the strict transform of the locus of rank $i$ matrices.

\begin{remark} \label{rmk:degenCQ}
	We can define degeneration spaces $S_r$ also for complete quadrics, either as the strict transform of the exceptional locus of the blow-up $X_{r-1} \longleftarrow X_r$ in (\ref{eq:blowupCQ}), or by replacing $\ComplCol(W)$ by $\CQ(W)$ in \Cref{eq:Sr}.
\end{remark}

\begin{remark}
	The space $\CQ(W)$ is the intersection of $\ComplCol(W)$ with ${\PP\left(S^2W\right)} \times \cdots \times {\PP\left(S^2(\bigwedge^{d-1}W)\right)}$.
\end{remark}

\subsection{Geometry of complete quadrics}
To a nonzero symmetric matrix $A \in S^2W$ we can associate a quadric hypersurface in $\PP(W^*)$, namely $$Q(A) := \{ x \in \PP(W^*) \mid x^T A x = 0\}.$$ Hence we can think of $\PP(S^2W)$ as the space of quadrics in $\PP(W^*)$.
Assume for a moment that $A$ is invertible. Then the nonzero matrix $\bigwedge^k A$ cuts out a quadric $$Q(\bigwedge^k A) = \{ x \in \PP(\bigwedge^k W^*) \mid x^T (\bigwedge^kA) x = 0\}$$ in $\PP(\bigwedge^k W^*)$. If we intersect $Q(\bigwedge^k A)$ with the Grassmannian $\GG(k-1,\PP(W^*)) \subseteq \PP(\bigwedge^k W^*)$, we obtain the locus of all $(k-1)$-planes tangent to $Q(A)$:%
\begin{lemma}
The intersection $Q(\bigwedge^k A) \cap \GG(k-1,\PP(W^*))$ is the space of $(k-1)$-planes tangent to $Q(A)$.
\end{lemma}
\begin{proof}
Recall that for a point $[w]$ on $Q(A)$, the tangent space at $[w]$ is equal to
\[
T_{[w]}Q(A) = \{x \in W \mid x^T \cdot A \cdot w = 0\}.
\]
Hence, the space $U=\langle u_1,\dots,u_k\rangle$ is tangent to the quadric $Q(A)$ if and only if there exist $p_1,\dots,p_k\in \CC$ not all equal to zero, such that
$u^T \cdot A \cdot (\sum_{i=1}^{k}{p_iu_i})=0$
for any $u\in U$.
This happens if and only if the $k\times k$ matrix with $(i,j)$-entry $u_i^T A u_j$ is degenerate, i.e.~the determinant of that matrix equals zero. But that determinant equals $(\bigwedge_{i=1}^k u_i)^T\cdot(\bigwedge^kA)\cdot(\bigwedge_{i=1}^k u_i),$ i.e.~the evaluation of the quadric $Q(\bigwedge^kA)$ on the point of the Grassmannian, $\bigwedge_{i=1}^k u_i$ corresponding to the space $U$.
\end{proof}
So a general point in $\CQ(W)$ is given by a smooth quadric together with its locus of tangent lines, its tangent planes, and so on. However, since in \Cref{def:CQ} we took a closure, there are other points in $\CQ(W)$.
\begin{example}\label{eg:completeConics1}
	Let $W=\CC^3$. For every $\varepsilon \in \CC\setminus {\{0\}}$, the point
	\[
	(A_1,A_2)
	=
	(
	\begin{pmatrix}
	-\varepsilon & 0 &0 \\ 0 & \varepsilon & 0 \\ 0 & 0 & 1
	\end{pmatrix},
	\begin{pmatrix}
	-1 & 0 &0 \\ 0 & 1 & 0 \\ 0 & 0 & \varepsilon
	\end{pmatrix}
	)
	\in \PP(S^2W) \times \PP(S^2(W^*))
	\]
	is contained in $\CQ(W)$, as $A_2$ equals the adjugate of $A_1$ up to scaling. Geometrically: $Q(A_1) \subset \PP(W^*)$ is the smooth conic with equation $\varepsilon x_0^2 = \varepsilon x_1^2 + x_2^2$. $Q(A_2) \subseteq \PP(W)$ is the dual conic; it is the space of lines tangent to $Q_1$, and has equation $\beta_0^2 = \beta_1^2 + \varepsilon \beta_2^2$, where $\beta_i$ are the coordinates on $W$.
	
	Taking the limit $\varepsilon \to 0$, we find that
	\[
	(B_1,B_2)=
	(
	\begin{pmatrix}
	0 & 0 &0 \\ 0 & 0 & 0 \\ 0 & 0 & 1
	\end{pmatrix},
	\begin{pmatrix}
	-1 & 0 &0 \\ 0 & 1 & 0 \\ 0 & 0 & 0
	\end{pmatrix}
	)
	\in \CQ(W).
	\]
		
	Now $Q(B_1)$ is the double line $x_2^2=0$ in $\PP(V^*)$, and $Q(B_2) \subset \PP(V)$ is a degenerate conic given by the equation $\beta_0^2=\beta_1^2$. It is the locus of all lines in $\PP(V^*)$ that pass through $[1:1:0]$ or $[1:-1:0]$. See also \Cref{fig:completeConicDegen}.
	
	Note that the quadric $Q(B_2)$ cannot be recovered from the double line $Q(B_1)$: we need the extra data of two marked points on $Q(B_1)$. This corresponds to the fact that $B_1 \in \PP(S^2 W)$ is in the locus that was blown up.
		\begin{figure}
			\centering
			\includegraphics[width=0.32\textwidth]{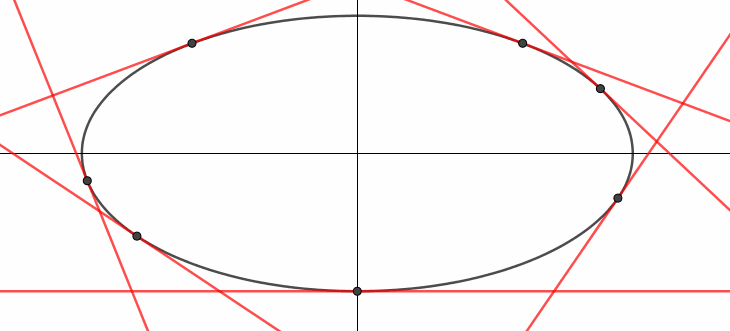}
			\includegraphics[width=0.32\textwidth]{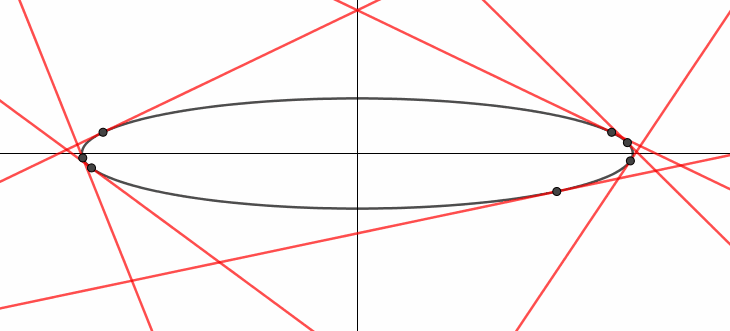}
			\includegraphics[width=0.32\textwidth]{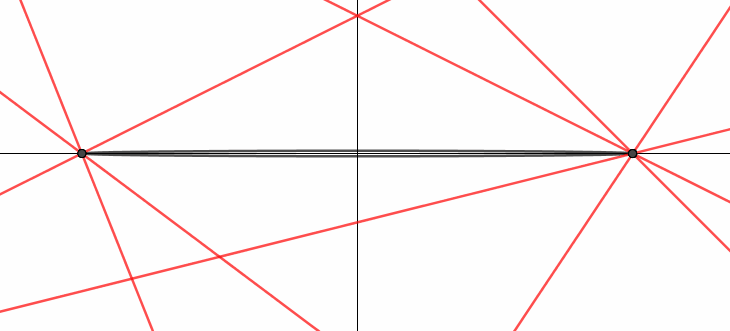}
			\caption{A conic degenerating to a double line with two marked points.}
			\label{fig:completeConicDegen}
		\end{figure}
	
	The complete conics can be classified into four types (see also \Cref{fig:completeConics4Types}):
	\begin{itemize}
		\item $A_1$ and $A_2$ both have rank $3$. Then $Q(A_1)$ is a smooth conic and $Q(A_2)$ is its dual.
		\item $A_1$ has rank $1$ and $A_2$ has rank $2$, as in the example above. Then $Q(A_1)$ is a double line and $Q(A_2)$ is the locus of all lines through 2 marked points on $Q(A_1)$.
		\item $A_1$ has rank $2$ and $A_2$ has rank $1$. This is dual to the case above. Then $Q(A_1)$ is a union of two lines, and $Q(A_2)$ is a double line in dual space.
		\item $A_1$ and $A_2$ both have rank $1$. Here $A_2$ is a rank one matrix that vanishes on the image of $A_1$, and $Q(A_1)$ and $Q(A_2)$ are both double lines.
	\end{itemize}
	\begin{figure}
		\centering
		\begin{tabular}{|r|l|} %
			\hline
			$\rk(A_1)=3, \rk(A_2)=3$ & $\rk(A_1)=1, \rk(A_2)=2$ \\
			\includegraphics[height=9em]{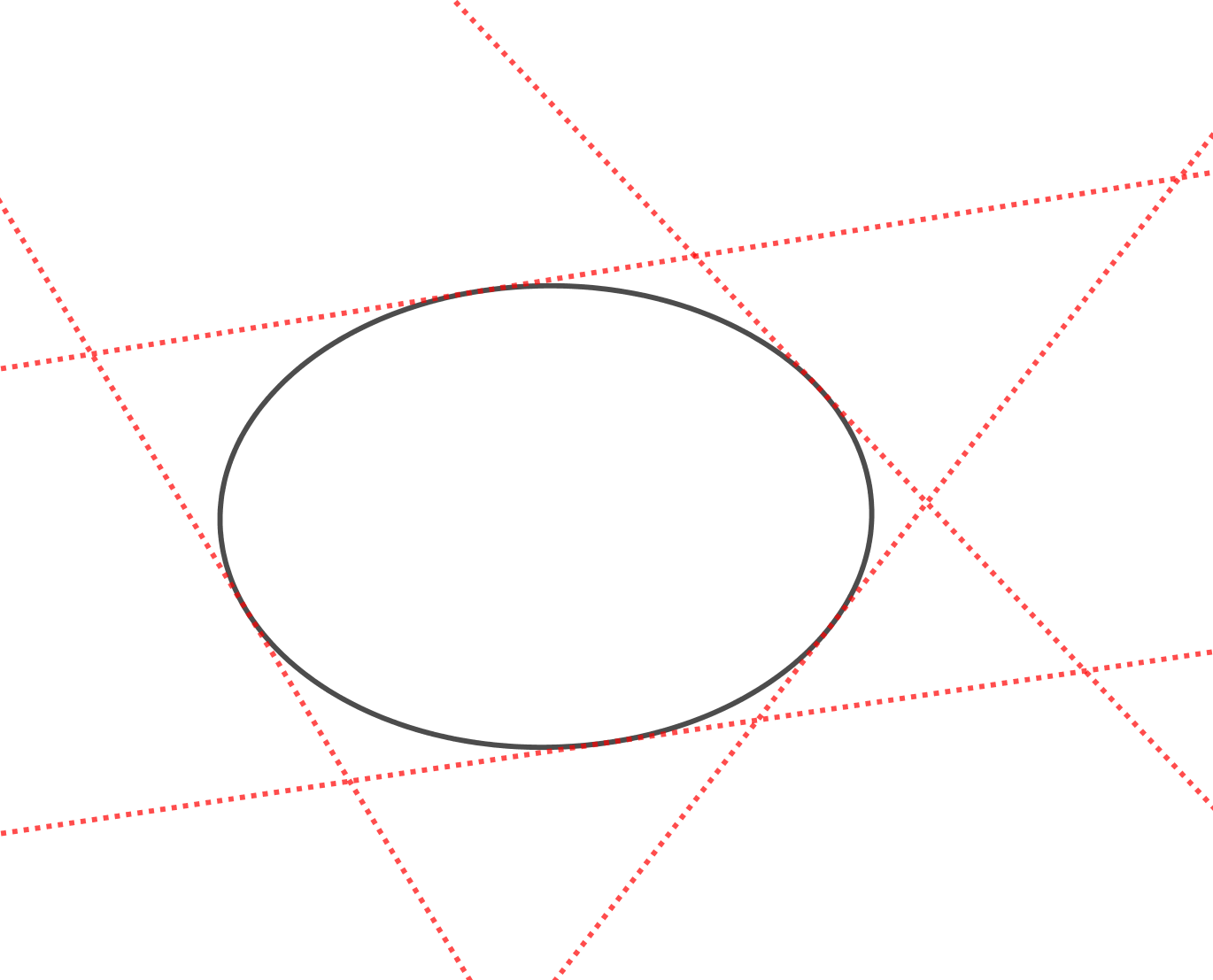} & \includegraphics[height=9em]{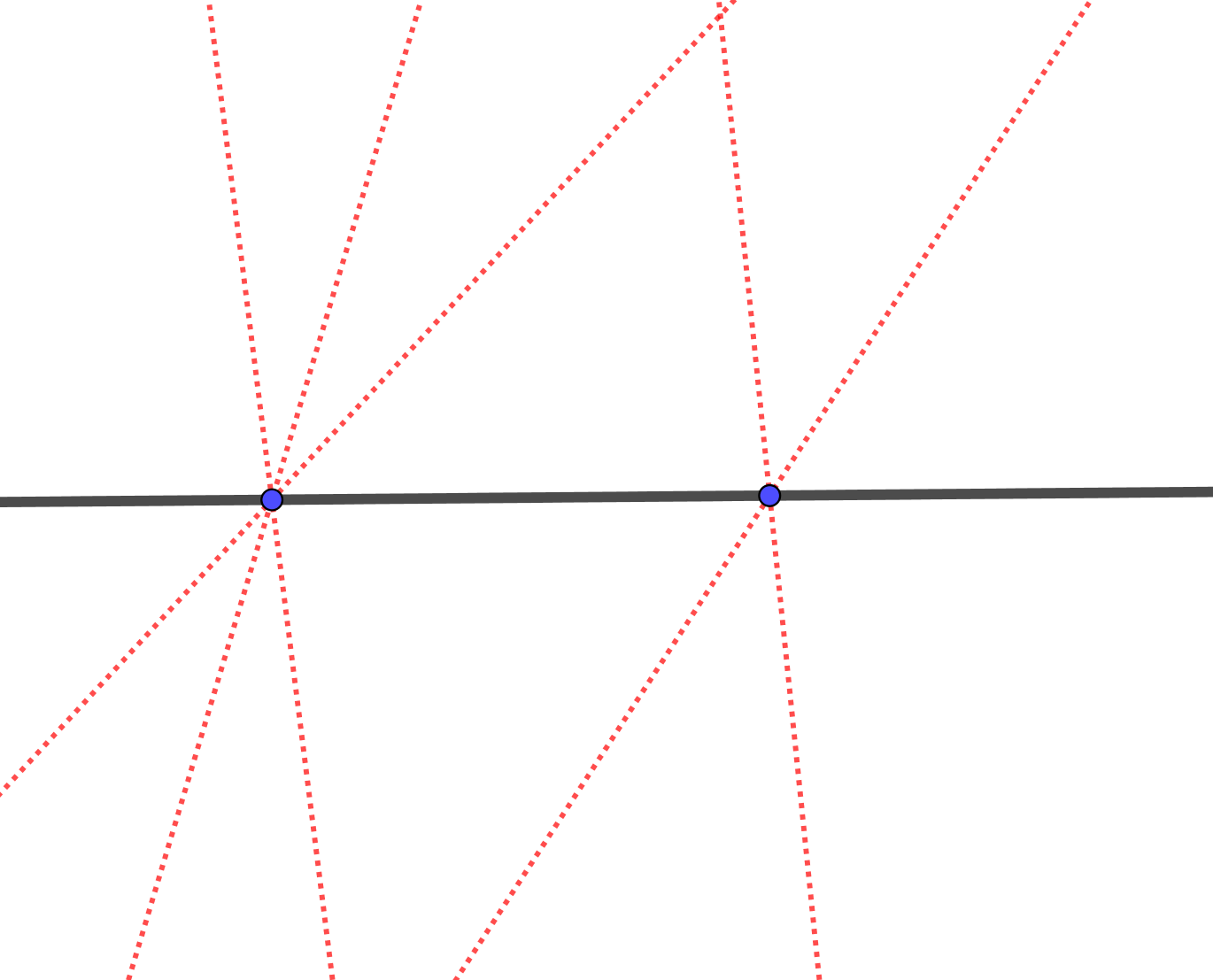} \\
			\hline
			$\rk(A_1)=2, \rk(A_2)=1$ & $\rk(A_1)=1, \rk(A_2)=1$ \\
			\includegraphics[height=9em]{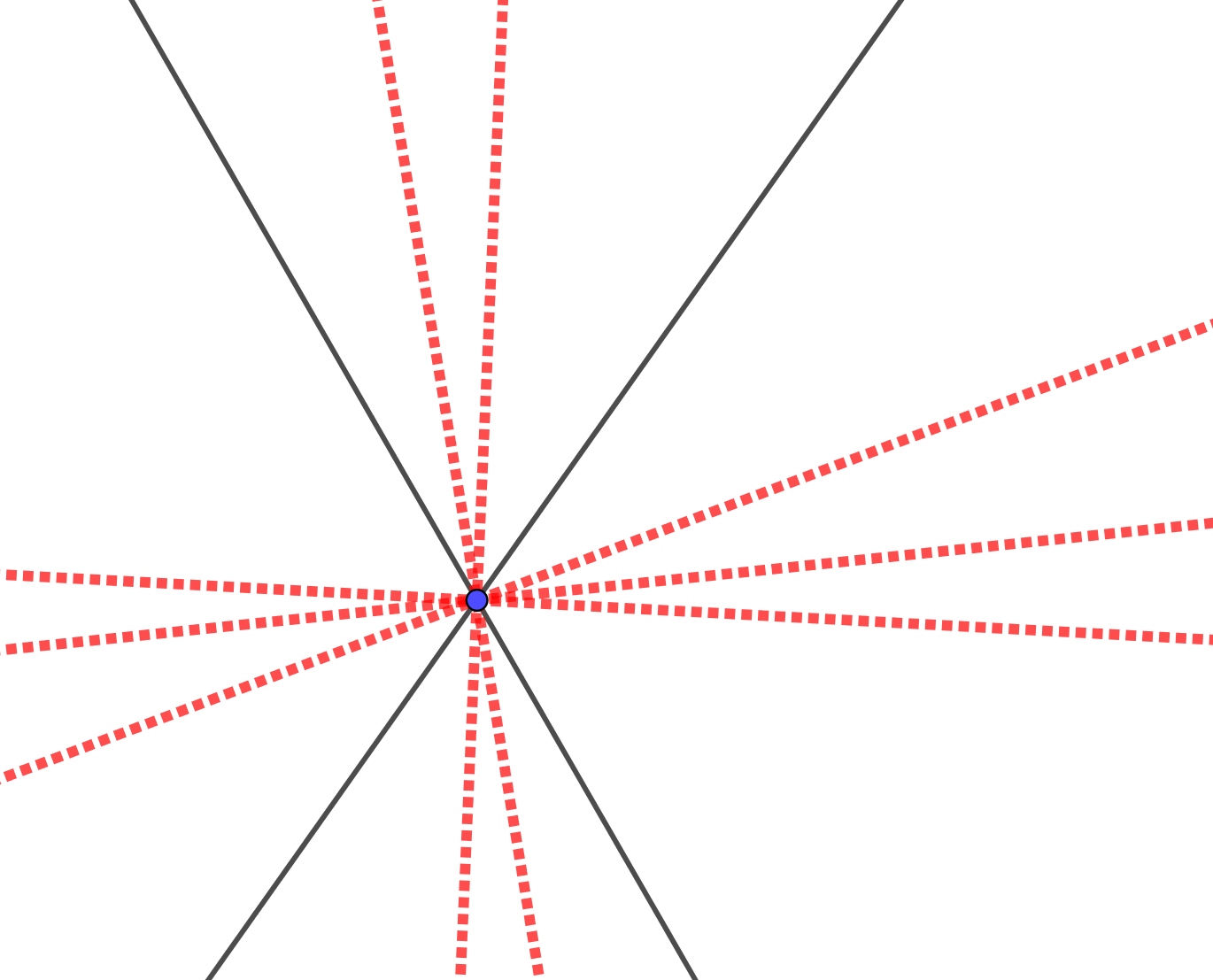} & \includegraphics[height=9em]{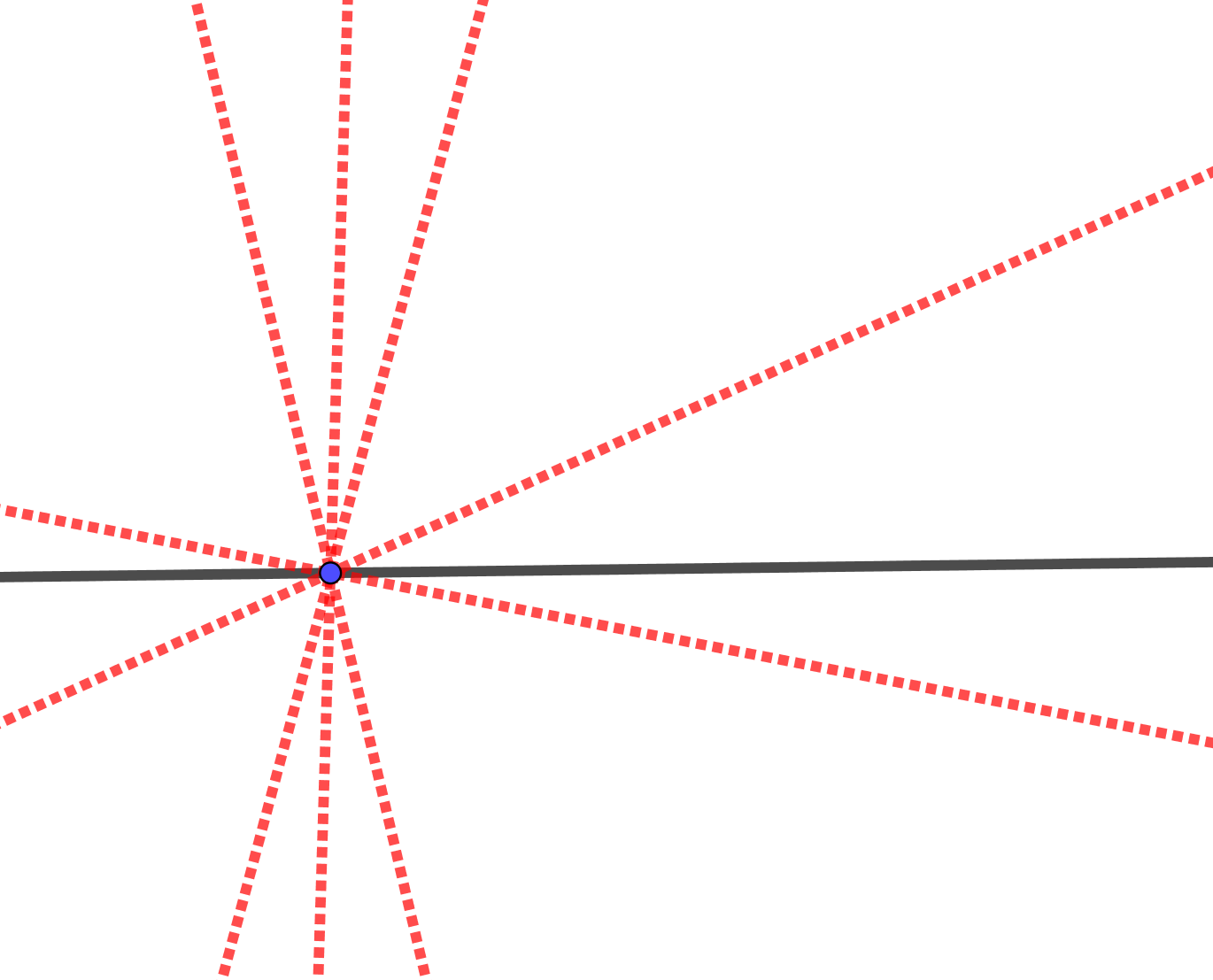}\\
			\hline
		\end{tabular}
		\caption{The four types of complete conics}
		\label{fig:completeConics4Types}
	\end{figure}
\end{example}

\begin{example}
	\begin{figure}
		\centering
		\includegraphics[width=\textwidth]{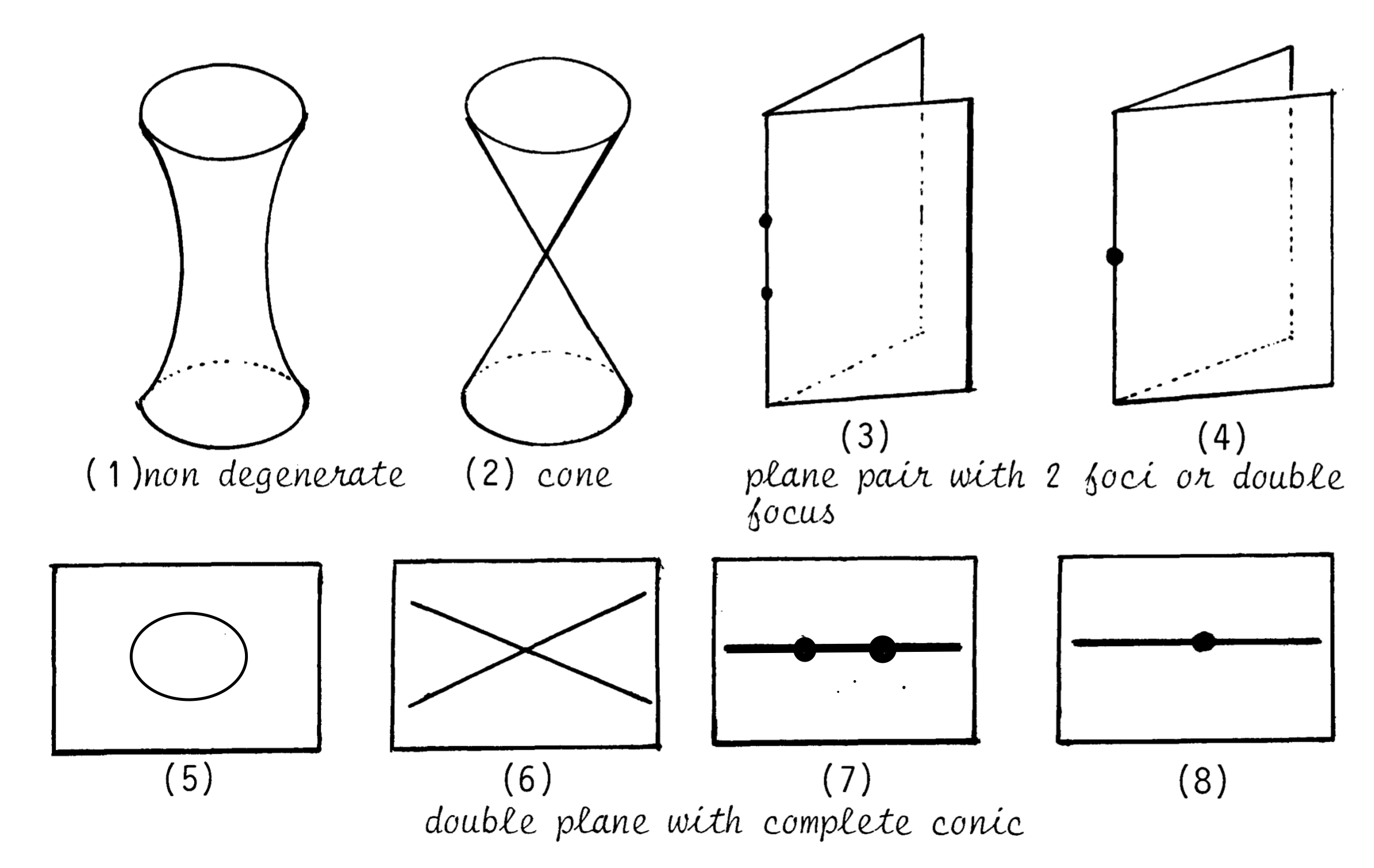}
		\caption{Complete quadrics in 3 dimensions. Picture taken from \cite{Vainsencher}.}
		\label{fig:CQ3D}
	\end{figure}
\Cref{fig:CQ3D} shows the classification of complete quadrics in the case $d=4$. More precisely, the pictures show $Q(A_1)$ and the additional data needed to determine $Q(A_2)$ and $Q(A_3)$.
\begin{enumerate}
	\item $\rk(A_1)=4,\rk(A_2)=4,\rk(A_3)=4$.
	\item $\rk(A_1)=3,\rk(A_2)=3,\rk(A_3)=1$. Here $Q(A_2)$ is the locus of lines tangent to the cone $Q(A_1)$, and $Q(A_3)$ is the double locus of all planes passing through the vertex.
	\item $\rk(A_1)=2,\rk(A_2)=1,\rk(A_3)=2$. Here $Q(A_2)$ is the double locus of all lines passing through the marked line, and $Q(A_3)$ is the locus of all planes passing through one of the marked points.
	\item $\rk(A_1)=2,\rk(A_2)=1,\rk(A_3)=1$. Again $Q(A_2)$ is the double locus of all lines passing through the marked line, and $Q(A_3)$ is the double locus of all planes passing through the marked point.
	\item $\rk(A_1)=1,\rk(A_2)=3,\rk(A_3)=3$. Here $Q(A_2)$ is the locus of all lines intersecting the marked conic, and $Q(A_3)$ is the locus of all planes tangent to the marked conic.
	\item $\rk(A_1)=1,\rk(A_2)=2,\rk(A_3)=1$. Again $Q(A_2)$ is the locus of all lines intersecting the marked degenerate conic, and $Q(A_3)$ is the double locus of all planes passing through the marked point.
	\item $\rk(A_1)=1,\rk(A_2)=1,\rk(A_3)=2$. Here $Q(A_2)$ is the double locus of all lines intersecting the marked line, and $Q(A_3)$ is the locus of all planes passing through one of the marked points.
	\item $\rk(A_1)=1,\rk(A_2)=1,\rk(A_3)=1$. Here $Q(A_2)$ is the double locus of all lines intersecting the marked line, and $Q(A_3)$ is the double locus of all planes passing through the marked point.
\end{enumerate}
\end{example}

\begin{prop} \label{prop:muGivesNumberOfQuadrics}
	The number of smooth quadric hypersurfaces in $\PP^{d-1}=\PP(W^*)$ passing through $a$ given general points and tangent to $b$ given general hypersurfaces (where $a+b=\binom{d+1}{2}-1$) is equal to $\mu_a(V)=\mu_b(V)$.
\end{prop}
\begin{proof}
	Let $A \subset \PP(S^2 W)$ be the subspace of all quadrics passing through our $a$ given points. This is a hyperplane of codimension $a$. Similarily, let $B \subset \PP(S^2 W^*)$ be the codimension $b$ hyperplane of all quadrics in dual space containing the $b$ given points in dual space.
	
	As usual, we will write $\Gamma$ for the graph of the inversion map $\PP(S^2 W) \dashrightarrow \PP(S^2 W^*)$, and $\pi_1$ and $\pi_2$ for the projections to $\PP(S^2 W)$ and $\PP(S^2 W^*)$ respectively. In this case, it is known that
	$\Gamma$ consists of all pairs $(C,D) \in \PP(S^2 W) \times \PP(S^2 W^*)$ of symmetric matrices for which $C\cdot D = \lambda I_d$ for some $\lambda \in \CC$. The smooth locus of $\Gamma$ are the pairs $(C,D)$ where $C$ and $D$ are invertible (and hence $D=C^{-1}$); the singular locus of $\Gamma$ consists of the pairs $(C,D) \in \Gamma$ where $C$ and $D$ are singular matrices (and hence $C\cdot D =0$).
	\begin{claim}
	The subvarieties $\pi_1^{-1}(A)$ and $\pi_2^{-1}(B)$ of $\Gamma$ intersect transversally, and all intersection points lie on the smooth locus of $\Gamma$.
	\end{claim}
	\begin{proof}
		We first show that the intersection is transverse on the smooth locus of $\Gamma$. We can rewrite our intersection as
		$$
		\pi_1^{-1}(H_1) \cap \ldots \cap \pi_1^{-1}(H_a) \cap \pi_2^{-1}(H'_1) \cap \ldots \cap \pi_2^{-1}(H'_b).
		$$
		Note that on the smooth locus of $\Gamma$, we have a transitive action of $GL(W)$, and moreover each $\pi_i^{-1}(H_j)$ is a general element of its orbit. Then the intersection is transverse by repeatedly applying Kleiman's transversality theorem (see \cite[Theorem 1.7]{eisenbud20163264}).
		
		To show that there are no intersection points in the singular locus, we observe that $\Gamma$ can be written as a union of locally closed subvarieties $\Gamma_{x,y}:=\{(C,D) \in \Gamma \mid \rk(C)=x, \rk(D)=y\}$, and $GL(W)$ acts transitively on each $\Gamma_{x,y}$. The smooth locus is exactly $\Gamma_{d,d}$; we fix $0< x,y <d$ and show there are no intersection points on $\Gamma_{x,y}$. Indeed: observe that every $\pi_i^{-1}(H_j) \cap \Gamma_{x,y}$ is a codimension one subvariety which is a general element of its $GL(W)$-orbit. Hence we can again apply Kleiman's transversality theorem, but this time our transverse intersection is empty for dimension reasons.
	\end{proof}
	Then by (\ref{eq:muIntersectionProduct}) and the claim:
	\begin{align*}	
	\mu_a(V) =& \int_{\Gamma}{H_1^a H_2^b} = \int_{\Gamma}{[\pi_1^{-1}(A)]\cdot[\pi_2^{-1}(B)]} \\ =& \int_{\Gamma}{[\pi_1^{-1}(A) \cap \pi_2^{-1}(B)]} = \# \left(\pi_1^{-1}(A) \cap \pi_2^{-1}(B)\right).
	\end{align*}
\end{proof}

\subsection{Complete skew-forms}
This section is completely analogous to \Cref{sec:CQ}, replacing symmetric matrices by skew-symmetric matrices. As the determinant of a skew-symmetric matrix of odd size is always zero, we now let $W$ be a vector space of even dimension $d=2e$, $V=\bigwedge^2W$, and $F = \pf \in S^e V^*$ the Pfaffian of a general skew-symmetric matrix.
\begin{remark}
	Since $\det = \pf^2 \in S^d V$, the maps $\grad(\det)$ and $\grad(\pf)$ agree as maps $\PP(V) \dashrightarrow \PP(V^*)$. So we could also just have taken $F=\det \in S^d V$.
\end{remark}
\begin{definition} \label{def:CS}
	The variety $\CS(W)$ of \emph{complete skew-forms} is the closure of the image of the set of invertible matrices under the map
	\begin{align*}
	\varphi: \PP(\bigwedge^2W) \dashrightarrow {\PP\left(\bigwedge^2W\right)} \times {\PP\left(\bigwedge^2(\bigwedge^2W)\right)} \times \cdots \times {\PP\left(\bigwedge^2(\bigwedge^{d-1}W)\right)},
	\end{align*}
	sending a matrix $A$ to $(A,\bigwedge^2 A, \ldots, \bigwedge^{d-1} A)$. It is a complete variety with respect to $F = \pf \in S^e V^*$.
\end{definition}

As in the symmetric case, the space of complete skew-forms can be constructed by blowing up $\PP(V)$
along the subvariety of rank two matrices, then the strict transform of the subvariety of matrices of
rank at most four, and so on.
As such, it admits a series $S_1,\ldots,S_{e-1}$ of special classes of divisors. %

For more on complete skew-forms, see \cite{Massarenti2}.

\subsection{General approach via partial derivatives} All of the examples in this section are a special case of the following construction: let $d=\deg F$, and for $k=1,\ldots,d-1$, let $D^k_1, \ldots,D^k_{m_k}$ be a basis of the span of all $k$th order partial derivatives of $F$, and let $\tilde{\Gamma}$ be the normalization of the image of the rational map
\begin{equation} \label{eq:partialDerivatives}
\Delta F: \PP(V) \to \PP(V) \times \PP(\bigwedge^2 V) \times \cdots \times \PP(\bigwedge^{d-1} V):
\end{equation}
\[
[x] \mapsto ([D^{d-1}_1(x):\ldots:D^{d-1}_{m_{d-1}}(x)],\ldots,[D^1_1(x):\ldots:D^{1}_{m_1}(x)]).
\]
In all our examples, it turned that $\tilde{\Gamma}$ is a complete variety. However, this is not the case for an arbitrary polynomial $F$. In \cite{stefana}, the question of when such a $\tilde{\Gamma}$ is smooth (and hence a complete variety) is considered, and some sufficient conditions are proven.

\begin{remark}
	Since $\PP(\bigwedge^k V) \cong \PP(\bigwedge^{d-k} V^*)$, we have an isomorphism
	\begin{align*}
	 \PP(V) \times \PP(\bigwedge^2 V) \times \cdots \times \PP(\bigwedge^{d-1} V) &\to \PP(V^*) \times \PP(\bigwedge^2 V^*) \times \cdots \times \PP(\bigwedge^{d-1} V^*) \\
	 (A_1,\ldots,A_{d-1}) & \mapsto (A_{d-1},\ldots,A_1).
	\end{align*}
	This map induces isomorphisms $\Per(V)\cong \Per(V^*)$, $\ComplCol(W) \cong \ComplCol(W^*)$, $\CQ(W) \cong \CQ(W^*)$, $\CS(W) \cong \CS(W^*)$.
\end{remark}

\section{Cohomology of the permutohedral variety}

\subsection{The permutohedral variety} \label{sec:perCohomology}
The aim of this subsection is to describe the Chow ring of the permutohedral variety $\Per_n$ introduced in Section \ref{subs:permutvar}. Our main tool will be toric geometry. Recall that an \emph{algebraic torus} is the multiplicative group $T_n:=(\CC^*)^n$. We will refer to $T_n$ simply as a torus.

A variety $X$ is called \emph{toric} if $T_n$ acts on $X$ and one of the orbits is dense. Depending on the context many authors also assume that $X$ is normal. We will not recall all the constructions of toric geometry as this is outside of the scope of this survey. The interested reader is referred to \cite{Cox, Fulton, michalek2018selected}, \cite[Chapter 8]{NonlinearAlgebra}. Here we just recall the most important facts about projective toric varieties $X$ associated to a lattice polytope $P\subset\RR^n$. We recall that a polytope is called a lattice polytope if all of its vertices belong to $\ZZ^n$.
\begin{enumerate}
\item The variety $X$ is the closure of the image of the monomial map, where the monomials correspond to lattice points in $P$. (Here there are two conventions, and following \cite{Fulton} one would first need to rescale $P$.)
\item There is a bijection between the $T_n$-orbits in $X$ and faces of $P$. The dimension of the orbit (as a complex algebraic variety) equals the dimension (as a real polytope) of the corresponding face.
\item The permutohedral variety $\Per_n$ is a smooth, projective, toric variety.
\item The polytope $\PPer_n$ associated to the permutohedral variety $\Per_n$ is known as the \emph{permutohedron}. It is the convex hull of the $(n+1)!$ lattice points $(\pi(1),\pi(2),\dots,\pi(n),\pi(n+1))\in\ZZ^{n+1}$ corresponding to the permutations $\pi\in S_{n+1}$.
\end{enumerate}
Our first aim is to provide a good understanding of the polytope $\PPer_n$. Let $[n+1]:=\{1,\dots,n+1\}$. A chain of strictly increasing subsets $\emptyset=A_0\subsetneq A_1\subsetneq\dots \subsetneq A_k=[n+1]$ where $A_i\subset [n+1]$ is called a \emph{flag} on $[n+1]$ of length $k$. A flag is called \emph{complete} if $k=n+1$.
\begin{example}\label{exm:permn=2}
For $n=2$ we have the following flags:
\begin{enumerate}
\item $\emptyset\subset [3]$, which corresponds to $k=1$.
\item Six distinct flags corresponding to $k=2$. For three of them $A_1$ is a singleton and for the other three $A_1$ has two elements.
\item Six complete flags, for example $\emptyset\subset \{2\}\subset \{2,3\}\subset \{1,2,3\}$.
\end{enumerate}
\end{example}
The faces of $\PPer_n$ are in bijection with flags on $[n+1]$, where a flag of length $k$ corresponds a face of dimension $n+1-k$. We first note that a complete flag $A_0\subsetneq A_1\subsetneq
\dots\subsetneq A_{n+1}$ corresponds to the permutation $\pi$ defined by $\{\pi(1)\}=A_1$, $\{\pi(1),\pi(2)\}=A_2,\dots,\{\pi(1),\dots,\pi(j)\}=A_j,\dots$.

A vertex corresponding to a complete flag $A_0\subsetneq A_1\subsetneq
\dots\subsetneq A_{n+1}$ belongs to a face corresponding to a flag $B_0\subsetneq B_1\subsetneq
\dots\subsetneq B_k$ if and only if the set $\{B_0,\dots,B_k\}$ is a subset of the set $\{A_0,\dots,A_{n+1}\}$, i.e.\ the flag $A_{\bullet}$ refines the flag $B_{\bullet}$.
\begin{example}
We continue Example~\ref{exm:permn=2}. Here $\PPer_2$ is a hexagon. The unique flag for $k=1$ in Example \ref{exm:permn=2} corresponds to the whole polytope. The six flags for $k=2$ are the six edges and the complete flags correspond to six vertices. This is presented in Figure \ref{fig:perm2}. %
\end{example}
\begin{figure}
\caption{Two dimensional permutohedron $\PPer_2$. The vertices are labelled by permutations, where the permuation $abc$ corresponds to the complete flag $\{a\}\subset\{a,b\}$. The edges are labelled by the set $A_1$ which uniquely defines the flag.}\label{fig:perm2}
\includegraphics[scale=0.3]{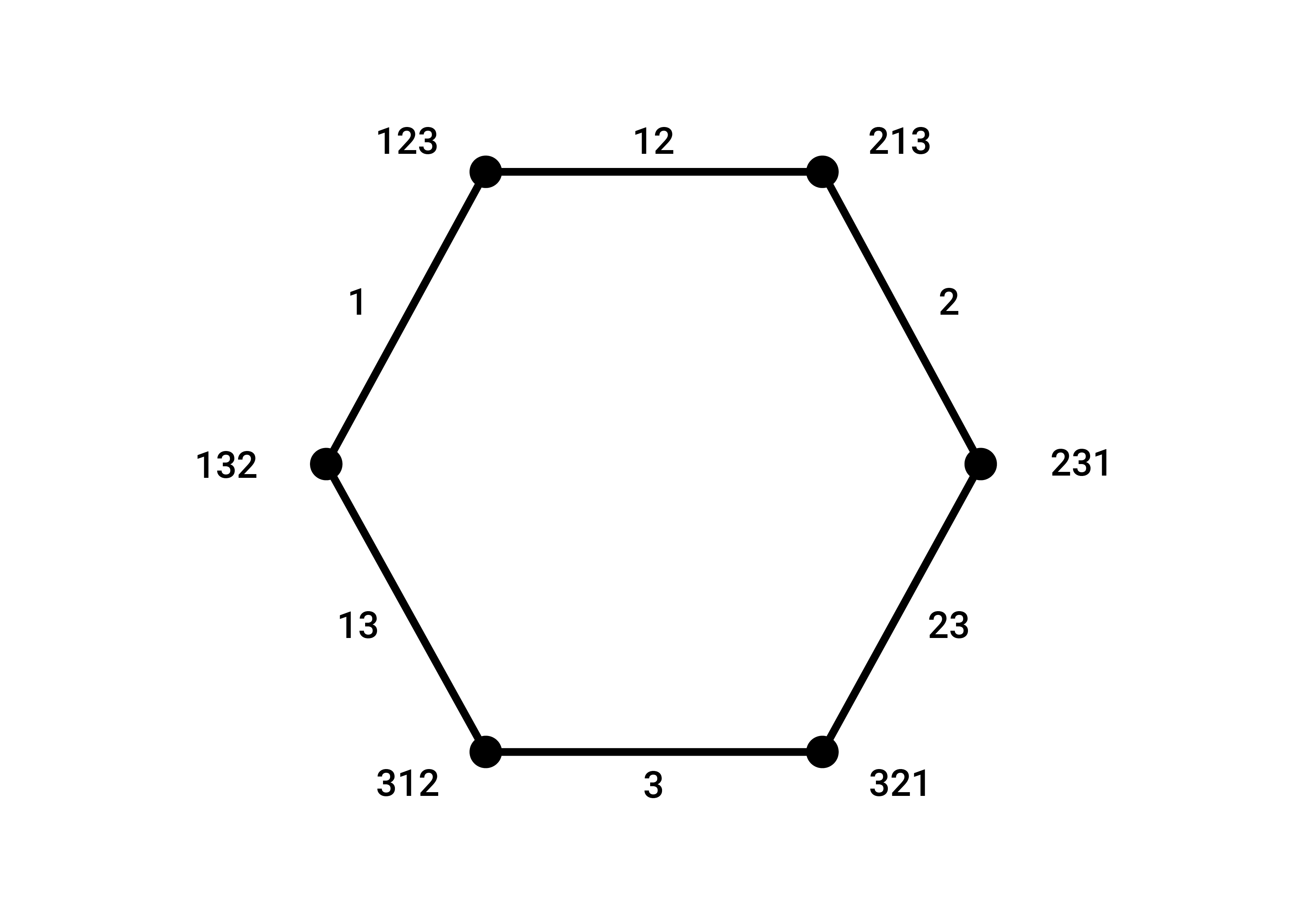}
\end{figure}

The big advantage of toric geometry is that in most cases it is enough to consider torus invariant structures. This is the case for the Chow ring. To avoid certain technicalities, as we are interested in $\Per_n$, we will assume that we deal with \emph{smooth}, projective, toric varieties and the Chow ring and cohomology ring are defined (i.e.~tensored) with $\QQ$.
As a vector space the Chow ring is linearly spanned by elements corresponding to torus orbit closures, i.e.~faces of the polytope. These are not linearly independent. Following \cite[Theorem 12.4.4]{Cox} we describe the ring structure.

The Chow ring is generated in degree one, i.e.~it is a polynomial ring $\QQ[x_1,\dots,x_s]$ where each $x_i$ corresponds to a given facet $F_i$ of the $n$ dimensional integral polytope $P\subset \RR^n$, modulo the ideal generated by the following relations:
\begin{enumerate}
\item $\prod_{i\in I} x_i=0$ whenever $\bigcap_{i\in I} F_i=\emptyset$,
\item \label{two} $\sum_{i=1}^s (n_{F_i})_j x_i=0$ for $j=1,\dots,n$ and $(n_{F_i})_j$ denotes the $j$-th coordinate of the primitive, integral normal vector $n_{F_i}\in\ZZ^n$ to the facet $F_i$.
\end{enumerate}
We note that our construction of $\PPer_n$ realizes it not as a full dimensional polytope, as the sum of coordinates is constant. For that reason, one often first shifts it to $0$ and considers it in the sublattice generated by the integral points inside $\PPer_n$.
\begin{example} \label{eg:Per2}
	We identify $\PPer_2$ with the hexagon that is the convex hull of $(0,0),(0,1),(1,2),(2,2),(2,1),(1,0)$. We have the six facets: $F_1=[(0,0),(0,1)]$, $F_2=[(0,1),(1,2)]$,$\dots$, $F_6=[(0,1),(0,0)]$ corresponding respectively to variables $x_1,x_{12},x_2,x_{23},x_3,x_{13}$.
	The generators of the ideal that are of the first type are: \[x_1x_2,x_1x_{23},x_1x_3,x_{12}x_{23},x_{12}x_3,x_{12}x_{13},x_2x_3,x_2x_{13},x_{23}x_{13}. \]
	The generators of the second type are:
	\[f=x_1+x_{12}-x_{23}-x_3,\quad g=-x_{12}-x_2+x_3+x_{13}.\]
	Hence the Chow ring equals:
	\[\frac{\QQ[x_1,x_{12},x_2,x_{23},x_3,x_{13}]}{(x_1x_2,x_1x_{23},x_1x_3,x_{12}x_{23},x_{12}x_3,x_{12}x_{13},x_2x_3,x_2x_{13},x_{23}x_{13}, f, g)}.\]
	Note that in this ring it holds that
	\begin{equation} \label{eq:topDegree}
	x_1x_{12}=x_2x_{12}=x_2x_{23}=x_{3}x_{23}=x_{3}x_{13}=x_1x_{13},
	\end{equation}
	and all other degree $2$ squarefree monomials are zero. So the map $\int_{\Per_2}$ identifies the top degree part of our Chow ring with $\QQ$ by sending each monomial in (\ref{eq:topDegree}) to $1$.
\end{example}
The first type of relations allows us to associate to any face $F$ a monomial $m_F:=\prod_{F\subset F_i} x_i$. The Chow ring, in degree $d$, is spanned by the monomials $m_F$ for $F$ of codimension $d$. The monomials $m_F$ are not independent in general, due to relations coming from point (\ref{two}). As one can see from the description above, the linear relations among $m_F$ are as follows:

 Fix a $(d+1)$-dimensional face $F$ and pass to the lattice spanned by $F$. We have $\sum_{F_i\subset F} (n_{F_i,F})_j m_{F_i}=0$, where the sum is over $d$-dimensional faces of $F$ and $n_{F_i,F}$ is the normal vector to $F_i$ in $F$.

\begin{example}
	Continuing \Cref{eg:Per2}: the permutohedral variety $\Per_2$ comes equipped with two maps $\pi, \pi': \Per_2 \to \PP^2$ as in (\ref{eq:diagramCompleteVariery}). The Chow ring of $\PP^2$ is isomorphic to $A^{\bullet}(\PP^2) = \QQ[\zeta]/(\zeta^3)$, where $\zeta$ is the hyperplane class. %
	The map $\pi$ of toric varieties corresponds to the map of polytopes that contracts the edges labeled $12$, $23$, $13$ to points, so we obtain the following triangle, which is the polytope associated to $\PP^2$:
	
	\begin{figure}[h]
		\begin{tikzpicture}[scale=2]
		\coordinate (a) at (0,0);
		\coordinate (b) at (1,0);
		\coordinate (c) at (0.5,1);
	
		\node (1) at (0.15,0.5) {1};
		\node (2) at (0.85,0.5) {2};
		\node (3) at (0.5,-0.15) {3};
		\node (4) at (1.2,0) {23};
		\node (5) at (-0.2,0) {13};
		\node (6) at (0.3,1) {12};
		\draw (a) -- (b) -- (c) --(a);
		\foreach \pt in {a,b,c} \fill[black] (\pt) circle (1pt);
		
		\end{tikzpicture}
	\end{figure}
	
	The hyperplane class in $\PP^2$ corresponds to an edge of this triangle, say the one labeled by $1$. The preimage of this edge is the union of the three edges in our hexagon labeled $1,12,13$. This means that the pullback of the hyperplane class is given by $H_1=\pi^*(\zeta)=x_1+x_{12}+x_{13}$.
	Analogously, $\pi'$ corresponds to contracting the edges $1,2,3$ of our hexagon, and we find that $H_2=\pi'^*(\zeta) = x_{23}+x_2+x_3$.
	
	We now compute the numbers $\mu_i := \mu_i(\PP^2)$ as in (\ref{eq:muChow}). For instance:
	$$
	\mu_1=\int_{\Per_2}{H_1H_2}=\int_{\Per_2}{(x_1+x_{12}+x_{13})(x_{23}+x_2+x_3)}=\int_{\Per_2}{x_2x_{12}+x_3x_{13}} =  2
	$$
	and similarly $\mu_0=\mu_2=1$. By \Cref{thm:HuhKatz}, these $\mu_i$ are the coefficients of the reduced characteristic polynomial of the matroid corresponding to $\CC^3 \subseteq \CC^3$, which is the uniform matroid of rank 3 on 3 elements.
\end{example}

\begin{remark}
 So far we have described elements of the Chow ring in the spirit of homology, that is as classes of subvarieties with coefficients. There is a dual description, in the spirit of cohomology, as functions on such classes. This was developed by Fulton and Sturmfels in \cite{fulton1997intersection}, through the theory of Minkowski weights. As this is the approach that was taken in e.g.~\cite{huh2012log} we next briefly describe it.
 We already know that the degree $d$ part $A^d(X)=F/R$ of the Chow ring is the vector space $F$ with basis given by (monomials associated to) codimension $d$-faces of the polytope, modulo the subspace $R$ generated by vectors corresponding to codimension $(d-1)$-faces.

 Let us now look at the dual vector space $(F/R)^*$. Canonically this is the kernel $K$ of the natural surjective restriction map $F^*\rightarrow R^*$. Hence, $\Hom_\QQ(A^d(X),\QQ)=K$. Clearly, the elements of $K$ are in particular elements of $F^*$, i.e.~linear functions on $F$. As a linear function is determined by its values on the basis, the elements of $F^*$ are naturally represented by functions from the set of codimension $d$-faces of the polytope to $\QQ$. Such a function $f$ belongs to $K$ if and only if for every codimension $(d-1)$-face $F$ we have: $\sum_{F_i\subset F} (n_{F_i,F}) f(F_i)=0$. Note that here we omit by the $j$-th coordinate. This is because a sum of each coordinate of given vectors equals zero if and only if the sum of those vectors is zero.

 The condition on the function $f$ described above is often referred to as the \emph{balancing condition} and a function satisfying it is called a \emph{Minkowski weight}. Hence one can also interpret the elements of the cohomology ring as Minkowski weights. The setting described above was probably first presented in \cite{mcmullen1996weights}. In current works it is most often expressed on the normal fan of the polytope.
\end{remark}
\section{Cohomology of complete quadrics}\label{completequad}
Let $V$ be an $n$-dimensional vector space.
The space of complete quadrics $\CQ(V)$ has two interesting series of divisors. The first series are the degeneracy classes $S_1, \ldots,S_{n-1}$ introduced in \Cref{rmk:degenCQ} (we slightly abuse notation and use $S_r$ both for the subvariety $S_r \subset \CQ(V)$ and for the corresponding class in $A^1(\CQ(V))$). The second series is denoted by $L_1, \ldots, L_{n-1}$, where $L_i := \pi_i^*([H]) \in A^1(\CQ(V))$, with $[H]$ the hyperplane class in $A^1(\PP(S^2\bigwedge^iV))$.

We will write $\phi(n,d)$ for the number $\mu_{d-1}(S^2V)$. In the light of \Cref{isgenericof,aresultof}, $\mu_{d-1}(S^2V)$ is actually the ML-degree of the linear concentration model associated to a generic $d$-dimensional subspace, as introduced in Section~\ref{section:linearConcentrationsModels}. In terms of the Chow ring:
\[
\phi(n,d) = \int_{CQ_n}{L_1^{\binom{n+1}{2} -d}L_2^{d-1}}.
\]

The next result was already know by Schubert~\cite{schubert1894allgemeine}, and for a modern approach, see \cite[Theorem 3.13]{Massarenti1}.
\begin{prop}\label{prop:relations}
	The classes $L_1,\ldots, L_{n-1}$ form a basis of $\Pic(\CQ(V))$, in which the classes
	$S_1, \ldots, S_{n-1}$ are given by the
	relations
	$$
	S_i= -L_{i-1}+2L_i-L_{i+1},
	$$
	with $L_0=L_n:=0$.
\end{prop}

As an immediate consequence, we find that for every subset $X$ of $\{1, \ldots, n-1\}$, the set $\{S_i \mid i \in X\} \sqcup \{L_i \mid i \notin X\}$ forms a $\QQ$-basis of $A^1(\CQ_n)$.

 The Chow rings of other complete varieties presented in the previous section, the complete skew-forms and the complete collineations, behave similarly to the Chow ring of the complete quadrics, and for this reason we will not treat them separately. The reader can consult \cite{Massarenti1, Massarenti2}.

\subsection{A general algorithm for computing intersection products} De Concini and Procesi \cite{DeConciniProcesi1} introduced an algorithm to compute arbitrary intersection products of the form
\begin{equation} \label{eq:CQgeneralProduct}
\int_{\CQ_n}{S_1^{a_1}\cdots S_{n-1}^{a_{n-1}} \cdot L_1^{b_1} \cdots L_{n-1}^{b_{n-1}}}.
\end{equation}
Here we assume that $\sum_{i=1}^{n-1} a_i+b_i= \dim (\CQ_n)=\binom {n+1}{2}-1$.
The idea is as follows:
\begin{enumerate}
	\item We can reduce to the case $a_1=\ldots=a_{n-1}=1$.
	\item The variety $S_1 \cap \cdots \cap S_{n-1}$ is isomorphic to the complete flag variety $Fl_n$ of complete flags is $\CC^n$. Hence our intersection product can be rewritten as
\begin{equation}\label{algorithm}
	\int_{Fl_n}{L_1^{b_1} \cdots L_{n-1}^{b_{n-1}}}.
\end{equation}
	\item The Chow ring of $Fl_n$ is known explicitly.
\end{enumerate}

The first and third item above deserve some elaboration:

\addtocontents{toc}{\SkipTocEntry}
\subsection*{Reduction to the flag variety}

Our reduction relies on the following lemma:
\begin{lemma}
	Suppose that in the expression (\ref{eq:CQgeneralProduct}), we have that every $a_i$ is either $0$ or $1$  with at least one being $0$, and moreover for each $a_i=0$ we also have $b_i=0$. Then the intersection product is equal to $0$.
\end{lemma}
\begin{proof}
	 The idea is to use the map from $S_{i_1} \cap \cdots \cap S_{i_k} \to Fl(i_1, \ldots, i_k;n)$ and the fact that the occurring classes $L_i$ are pullbacks of classes on $Fl(i_1, \ldots, i_k;n)$. The corresponding product on $Fl(i_1, \ldots, i_k;n)$ is zero for dimension reasons. See also \cite[Section 9.2]{DeConciniProcesi1}.
\end{proof}
From now on, we will (without loss of generality) only consider expressions (\ref{eq:CQgeneralProduct}) with all $a_i$ either $0$ or $1$.
Assume we have given such an expression, we may assume that not all $a_i$ are equal to $1$ (else we are done) and that there is an $i$ with $a_i=0$ and $b_i \neq 0$ (else the product is $0$ by the lemma). Pick such an $i$, and rewrite one factor $L_i$ in terms of the basis $\{S_j \mid a_i=0\} \sqcup \{L_j \mid a_j=1\}$.
If we partially order all tuples $T=(a_1, \ldots, a_{n-1}, b_1, \ldots, b_{n-1})$ by $T \leq T'$ if and only if either $\sum{a_i} > \sum{a_i'}$, or $\sum{a_i}=\sum{a_i'}$ and $\sum_{i: a_i=0}{b_i} \leq \sum_{i: a_i'=0}{b_i'}$, then this expands our given intersection product into expressions that are smaller with respect to our partial order.

\addtocontents{toc}{\SkipTocEntry}
\subsection*{Intersection products on flag varieties}

We here present only the bare minimum necessary for our purposes, for more see \cite{manivel2001symmetric}, or the survey article \cite{Gillespie}.

The Chow ring of $Fl_n$ has a basis indexed by permutations of $\{1,\ldots,n\}$. Fixing a complete flag $0=F_0 \subset \cdots \subset F_n = \CC^n$, for such a permutation $w$, we define the \emph{Schubert class} $\sigma_w$ as the class of the \emph{closed Schubert cell}
\[
Y_w :=\overline{\{V_{\bullet} \mid \dim(V_p \cap F_q) = \#\{i \leq p \mid n+1-w(i) \leq q\} \quad \forall p,q \}}.
\]
One checks that the codimension of $Y_w$ is equal to the inversion number $inv(w)$ of $w$ (the number of pairs $i < j$ with $w(i) > w(j)$, or equivalently the minimal number of elementary transpositions $s_i=(i, i+1)$ needed to express $w$). So $\sigma_w \in A^{inv(w)}(Fl_n)$.
In the case $w = s_i$ is an elementary transposition, our Schubert cell becomes
\[
Y_{s_i} = \{V_{\bullet} \mid \dim(V_i \cap F_{n-i}) \geq 1 \}.
\]
So identifying $S_1 \cap \cdots \cap S_{n-1}$ with $Fl_n$, our class $L_i$ corresponds to the Schubert class $\sigma_{s_{n-i}}$. Products of such classes in $A^\bullet(Fl_n)$ can be computed using \emph{Monk's rule}:
\begin{theorem}
	\[
	\sigma_{s_i}\cdot\sigma_w= \sum{\sigma_v},
	\]
\end{theorem}
where the sum is over all permutations $v$ obtained from $w$ by:
\begin{itemize}
	\item Choosing a pair $p,q$ of indices with $p \leq i < q$ for which $w(p)<w(q)$ and for any $k$ between $p$ and $q$, $w(k)$ is not between $w(p)$ and $w(q)$.
	\item Defining $v(p)=w(q)$, $v(q)=w(p)$, and for all other $k$, $v(k)=w(k)$.
\end{itemize}
By repeatedly applying Monk's Rule, we can compute our product (\ref{algorithm}), finishing the algorithm.

\subsection{Polynomiality results}\label{sec:polynomiality}
In this section we will outline a proof of the following polynomiality result, which was conjectured by Sturmfels and Uhler in \cite{sturmfels2010multivariate}.
\begin{theorem}[{\cite[Theorem 4.2]{MMMSV}}] \label{thm:polyPhi}
	For fixed $d$, $\phi(n,d)$ is a polynomial in $n$, of degree $d-1$.
\end{theorem}
Although the algorithm from the previous section can in principle be used to compute each number $\phi(n,d)$, it is not clear how to obtain polynomiality using this approach. We will use a different strategy instead: in the previous section we pushed forward our intersection product to $\bigcap_{i}{S_i} \cong Fl_n$, here we will push it forward to the individual $S_r$, and then further to the Grassmannians $G(r,n)$.
\begin{definition}
	Let $0 < m < \binom{n+1}{2}$ and $0 < r < n$, then we define
	$$
	\delta(m,n,r) = \int_{CQ_n}S_rL_1^{\binom{n+1}{2}-m-1}L_{n-1}^{m-1}.
	$$
\end{definition}
\begin{remark}
	The number $\delta(m,n,r)$ is known as the \emph{algebraic degree of semidefinite programming} \cite{NRS,BothmerRanestad}. It is equal to the degree of the variety dual to the locus of rank $\leq r$ matrices in a general $m$-dimensional subspace of $\PP(S^2 \CC^n)$. The sum $\sum_{r}{\delta(m,n,r)}$ measures the algebraic complexity of the semidefinite programming problem
	\begin{equation}
		\text{maximize} \tr(B \cdot Y) \text{ subject to } Y \in L \text{ and } Y \text { is semipositive definite,}
	\end{equation}
	where $B \in S^2(\RR^n)$ is a generic $n \times n$-matrix, and $L \subseteq S^2(\RR^n)$ is a generic linear subspace of dimension $m$.
\end{remark}
\begin{prop}[{Pataki's inequalities, see \cite[Proposition 5]{NRS}}]
	$\delta(m,n,r) \neq 0$ if and only if
	\[
	\binom{n-r+1}{2} \leq m \leq \binom{n+1}{2} - \binom{r+1}{2}.
	\]
\end{prop}
Using the relations from \Cref{prop:relations} and the Pataki inequalities, we can write
\begin{equation}
\phi(n,d)=\frac{1}{n}\sum_{1 \leq \binom{s+1}{2} \leq d}{s\delta(d,n,n-s)}.
\end{equation}
So our \Cref{thm:polyPhi} follows immediately once we can prove the following:
\begin{theorem}[{\cite[Theorem 4.1]{MMMSV}}] \label{thm:polyDelta}
	For any fixed $m,s >0$, the function $\delta(m,n,n-s)$ is a polynomial in $n$, of degree $m$. Moreover this polynomial  vanishes at $n=0$. %
\end{theorem}
\begin{proof}[Sketch of proof]
	We have
	$$
	\delta(m,n,r) = \int_{S_r}L_1^{\binom{n+1}{2}-m-1}L_{n-1}^{m-1}.
	$$
	We can pushforward this computation to the Grassmannian $G(r,n)$, and get the following expression in terms of Segre classes:
	\[
	\delta(m,n,r) = \int_{G(r,n)}s_{\left(\binom{n+1}{2}-m-\binom{r+1}{2}\right)}(S^2\fU)s_{\left(m-\binom{n-r+1}{2}\right)}(S^2\fQ^*).
	\]
	Expanding these Segre classes yields a formula for $\delta$ in terms of the so-called \emph{Lascoux coefficients} (\cite[Theorem 1.1]{BothmerRanestad}, see also \cite{LaksovLascouxThorup,Pragacz}). So we are left with showing a certain polynomiality result for these Lascoux coefficients, which can be done combinatorially \cite[Theorem 4.3]{MMMSV}.
\end{proof}

We conclude this section with a slight generalization of the above polynomiality result.
\begin{definition}
	For $\binom{n+1}{2} >d$, we define
	$$
	\phi_c(n,d) = \int_{CQ_n}L_cL_1^{\binom{n+1}{2}-d-1}L_{n-1}^{d-1}.
	$$
\end{definition}
\begin{corollary}
	If $\binom{n-c+2}{2} > d$, then $\phi_c(n,d)=c \cdot \phi(n,d)$. In particular, in this range $\phi_c(n,d)$ is given by a polynomial in $n$.
\end{corollary}
\begin{proof}
	We can use \Cref{prop:relations} to write the class $L_c$ in terms of $S_i$'s:
	\begin{align*}
	L_c &= \sum_{j=1}^{n-1}{M_{c,j}S_j}, \text{ where } M_{c,j}=\min\{c,j\} - \frac{cj}{n} \\
	&= \sum_{j=1}^{c-1}{\frac{j(n-c)}{n}S_j} + \sum_{j=c}^{n-1}{\frac{c(n-j)}{n}S_j}.
	\end{align*}
	So we find
	\begin{align*}
	\phi_c(n,d) &= \sum_{j=1}^{c-1}{\frac{j(n-c)}{n}\delta(d,n,j)} + \sum_{j=c}^{n-1}{\frac{c(n-j)}{n}\delta(d,n,j)}\\
	&= \sum_{i=n-c+1}^{n-1}{\frac{(n-i)(n-c)}{n}\delta(d,n,n-i)} + \sum_{i=1}^{n-c}{\frac{ci}{n}\delta(d,n,n-i)}.
	\end{align*}
	By our assumption $\binom{n-c+2}{2} > d$ and Pataki's inequalities, every term in the first sum vanishes, and we obtain
	\[
	\phi_c(n,d):=c \cdot \sum_{\binom{i+1}{2} \leq d}{\frac{i}{n}\delta(d,n,n-i)} = c \cdot \phi(n,d).
	\]
\end{proof}

\begin{remark} \label{rmk:manivel2}
	In \cite{STZ}, it was conjectured that the analogue of \Cref{thm:polyPhi} also holds for the ML-degree of the linear \emph{covariance} model (see \Cref{rmk:manivel1}), together with explicit formulas in the cases $d=2,3,4$. This conjecture and formulas have been proven for $d=2$ in \cite{CoonsMariglianoRuddy}, and very recently for $d=3,4$ in \cite{manivel2021proof}. The latter proof also uses intersection theory on the space of complete quadrics, but needs some additional intricate geometrical arguments that do not obviously generalize to higher $d$.
\end{remark}

\subsection{The whole Chow ring}

So far, we have only dealt with classes in $A^\bullet(\CQ_n)$ that can be expressed in terms of $L_i$ and $S_i$, i.e.\ that are in the subring generated by $A^1(\CQ_n)$. If we want to describe the whole Chow ring, one approach is to construct an affine stratification of $\CQ_n$. The classes of the closed strata then form a basis of $A^\bullet(\CQ_n)$. One such stratification is given in \cite{Strickland}. We will now give an alternative description of this stratification, as a Bia{\l}ynicki-Birula decomposition.

If we choose a torus $(\CC^*)^n \cong T \subset S^2 V$, this gives rise to a $T$-action on $\CQ_n$. The $T$-fixed points can be indexed by \emph{$2$-permutations}, see \cite[Proposition 4.7]{mateusz1}.
\begin{definition}
	A \emph{$2$-permutation} of $[n]$ is a map $\sigma: [n] \to [k]$ for some $k \leq n$, such that for every $j \in [k]$, we have $\#\sigma^{-1}(j) \in \{1,2\}$.
\end{definition}

If we fix a general enough $\CC^*\subset T$, we get a decomposition of $\CQ_n$ into cells: two points belong to the same cell if and only if the limits (for $t \to 0$) of their $\CC^*$-orbits are the same. Theorem~4.4 of Bia{\l}ynicki-Birula \cite{BB} states that these cells give an affine stratification of $\CQ_n$. In particular, the classes of the closed strata form a basis of the Chow group. We make a particular choice of $\CC^* \subset T$, namely given by the inclusion $t \mapsto (t^{d_1}, \ldots, t^{d_n})$, where $d_1, \ldots, d_n$ are fixed integers satisfying
$$
2d_1 < d_1+d_2 < 2d_2 < d_1+d_3 < d_2+d_3 < 2d_3 < d_1 + d_4 < \ldots < 2d_n.
$$
 After making this choice, we can describe the dimensions of the cells, using the description of the $T$-action in \cite[Proposition 4.9]{mateusz1}. %
\begin{definition} \label{def:weights}
	The \emph{weight} of a $2$-permutation $\sigma: [n] \to [k]$ is equal to
	\begin{align*}
	w(\sigma) = & \#\{(i,j) \in [n] \times [n] \mid i < j \text{ and } \sigma(i) \leq \sigma(j) \}\\ & + \#\{j \in [k-1] \mid \max \sigma^{-1}(j) < \max \sigma^{-1}(j+1)\}.
	\end{align*}
\end{definition}
Then the dimension of the cell corresponding to a $2$-permutation $\sigma$ is equal to $w(\sigma)$. In particular, the dimension of the Chow group $A_m(\CQ_n)$ is equal to the number of $2$-permutations of weight $m$. This is \cite[Theorem 2.7]{Strickland}.

All of the cells admit an affine parametrization, which we will now describe. Our parametrization is based on the one in \cite{LaksovCQ}.
To a $2$-permutation $\sigma: [n] \to [k]$, we associate a symmetric $n \times n$ matrix $Y_{\sigma}$ whose entries are monomials in $k$ variables $y_1, \ldots, y_k$. The construction is as follows: for every $i \in [k]$, there are 2 possibilities: either $\sigma^{-1}(\{i\}) = \{j\}$, in which case we put $Y_{j,j}=y_1 \cdots y_i$, or $\sigma^{-1}(\{i\}) = \{j_1,j_2\}$, in which case we put $Y_{j_1,j_2}=Y_{j_2,j_1}=y_1 \cdots y_i$.
Furthermore, we let $X$ be the matrix
$$
\begin{pmatrix}
1 & 0 & 0 & \cdots & 0 \\
x_{12} & 1 & 0 & \cdots & 0 \\
x_{13} & x_{23} & 1 & \cdots & 0 \\
\vdots & \vdots & \vdots & \ddots & \vdots \\
x_{1n} & x_{2n} & x_{3n}  & \ldots & 1
\end{pmatrix}.
$$

We then consider the following map, which a priori looks like a rational map, but upon closer inspection turns out to be defined everywhere, i.e.\ is a morphism:
\begin{align*}
\rho_{\sigma}: \CC^{\binom{n}{2}} \times \CC^{k} &\to \CQ_n \\
(X,Y_\sigma) &\mapsto \left(X\cdot Y_{\sigma}\cdot X^T, \bigwedge^2 (X\cdot Y_{\sigma}\cdot X^T), \ldots, \bigwedge^{n-1} (X\cdot Y_{\sigma}\cdot X^T)\right).
\end{align*}
Our parametrization of the cell corresponding to $\sigma$ is a restriction of $\rho_{\sigma}$, obtained by putting some of the variables equal to $0$. More precisely, we put $x_{ij}=0$ if $i<j$ but $\sigma(i) > \sigma(j)$, and we put $y_j=0$ if $\max \sigma^{-1}(j) > \max \sigma^{-1}(j+1)$.
\begin{example}
	Recall that points in $\CQ_3$ are given by pairs of matrices $(A,B) \in \PP(S^2 \CC^3) \times \PP(S^2 \CC^3)$ satisfying $AB=\lambda\cdot I_3$ for some $\lambda \in \CC$. The torus action is given by $$\boldsymbol{t} \cdot (A,B) = (\begin{pmatrix}
	t_1 & 0 & 0 \\ 	0 & t_2 & 0 \\ 0 & 0 & t_3
	\end{pmatrix} \cdot A \cdot \begin{pmatrix}
	t_1 & 0 & 0 \\ 	0 & t_2 & 0 \\ 0 & 0 & t_3
	\end{pmatrix}, \begin{pmatrix}
	t_1^{-1} & 0 & 0 \\ 	0 & t_2^{-1} & 0 \\ 0 & 0 & t_3^{-1}
	\end{pmatrix} \cdot B \cdot \begin{pmatrix}
	t_1^{-1} & 0 & 0 \\ 	0 & t_2^{-1} & 0 \\ 0 & 0 & t_3^{-1}
	\end{pmatrix}).$$
	There are 12 $T$-fixed points, which are listed in \Cref{TheTable}. The corresponding $2$-permutations are also listed: here for instance $13|2$ means the map %
	\begin{align*}
	\{1,2,3\} &\to \{1,2\}\\ 1,3 &\mapsto 1\\ 2 &\mapsto 2.
	\end{align*}
	Next, we fix a $\CC^* \hookrightarrow T$ via $t \mapsto (1,t,t^3)$. Then we have
	$$
	t \cdot (\begin{pmatrix}
	a_{11} & a_{12} & a_{13} \\ a_{12} & a_{22} & a_{23} \\ a_{13} & a_{23} & a_{33}
	\end{pmatrix}, \begin{pmatrix}
	b_{11} & b_{12} & b_{13} \\ b_{12} & b_{22} & b_{23} \\ b_{13} & b_{23} & b_{33}
	\end{pmatrix}) = (\begin{pmatrix}
	a_{11} & ta_{12} & t^3a_{13} \\ ta_{12} & t^2a_{22} & t^4a_{23} \\ t^3a_{13} & t^4a_{23} & t^6a_{33}
	\end{pmatrix}, \begin{pmatrix}
	b_{11} & t^{-1}b_{12} & t^{-3}b_{13} \\ t^{-1}b_{12} & t^{-2}b_{22} & t^{-4}b_{23} \\ t^{-3}b_{13} & t^{-4}b_{23} & t^{-6}b_{33}
	\end{pmatrix}).
	$$
	We can now compute for each point $(A,B) \in \CQ_3$ the limit of the corresponding $\CC^*$, and hence we obtain the Bia{\l}ynicki-Birula cells. See \Cref{TheTable}. Note that the dimensions of these cells agree with the weights in \Cref{def:weights}. See also \cite[Example on p. 247]{Strickland}.
	
	All of the cells admit an affine parametrization. We will illustrate this for the $5$-dimensional cell, which corresponds to $\sigma = 1 | 2 | 3$. Define
	\[
	X= \begin{pmatrix}
	1 &0& 0 \\ x_{12} & 1 & 0 \\ x_{13} & x_{23} & 1
	\end{pmatrix}, \quad
	 Y= \begin{pmatrix}
	 1 &0& 0 \\ 0 & y_1 & 0 \\ 0 & 0 & y_1y_2
	 \end{pmatrix}, \quad
	 \tilde{Y} = \begin{pmatrix}
	 y_1y_2 &0& 0 \\ 0 & y_2 & 0 \\ 0 & 0 & 1
	 \end{pmatrix}.
	\]
	Then the parametrization of the cell is given by
	\begin{align*}
	\left(X \cdot Y \cdot X^T, \adj(X)^T \cdot \tilde{Y} \cdot \adj(X)\right)=(A,B)
	\end{align*}
	where
	\begin{align*}
	A= \begin{pmatrix}
	 1 &  x_{12} &   x_{13} \\
	x_{12} & x_{12}^2+y_1 &     x_{12}x_{13}+x_{23}y_1      \\
	x_{13} & x_{12}x_{13}+x_{23}y_1 &x_{23}^2y_1+x_{13}^2+y_1y_2
	\end{pmatrix}
	\end{align*}
	and
	\begin{align*}
	B=\begin{pmatrix}  x_{12}^2x_{23}^2-2x_{12}x_{13}x_{23}+x_{12}^2y_2+x_{13}^2+y_1y_2 & -x_{12}x_{23}^2+x_{13}x_{23}-x_{12}y_2 & x_{12}x_{23}-x_{13} \\
	-x_{12}x_{23}^2+x_{13}x_{23}-x_{12}y_2 &                   x_{23}^2+y_2 & -x_{23} \\
	x_{12}x_{23}-x_{13}  & -x_{23} &  1 \end{pmatrix}.
	\end{align*}
	For the 3-dimensional cell corresponding to $\sigma=2|13$, we write
	\[
	X= \begin{pmatrix}
	1 &0& 0 \\ x_{12} & 1 & 0 \\ x_{13} & x_{23} & 1
	\end{pmatrix}, \quad
	Y= \begin{pmatrix}
	0 &0& y_1 \\ 0 & 1 & 0 \\ y_1 & 0 & 0
	\end{pmatrix}, \quad
	\tilde{Y} = \begin{pmatrix}
	0 &0& 1 \\ 0 & y_1 & 0 \\ 1 & 0 & 0
	\end{pmatrix}.
	\]
	Then the parametrization of the cell is given by
	\begin{align*}
	\left(X \cdot Y \cdot X^T, \adj(X)^T \cdot \tilde{Y} \cdot \adj(X)\right)=(\begin{pmatrix}
	 0 & 0 &  y_1           \\
	 0 & 1 &  x_{23}          \\
	 y_1 & x_{23} & x_{23}^2+2x_{13}y_1 \\
	\end{pmatrix},\begin{pmatrix}
	 -2x_{13} & -x_{23} & 1 \\
	 -x_{23} & y_1  & 0 \\
	 1  &   0  &  0 \\
	\end{pmatrix}).
	\end{align*}
\begin{table}
	\centering
	\begin{tabular}{|l|l|l|l|l|}
		\hline
		$T$-fixed point & $2$-permutation & Cell & Dimension \\ \hline
		$(\begin{pmatrix}
			1 & 0 & 0 \\ 0& 0& 0 \\ 0 & 0 & 0
		\end{pmatrix}, \begin{pmatrix}
		0 & 0 & 0 \\ 0& 0& 0 \\ 0 & 0 & 1
		\end{pmatrix}) $ & $1|2|3$
		& $(\begin{pmatrix}
			1 & a_{12} & a_{13} \\ a_{12} & a_{22} & a_{23} \\ a_{13} & a_{23} & a_{33}
		\end{pmatrix}, \begin{pmatrix}
			b_{11} & b_{12} & b_{13} \\ b_{12} & b_{22} & b_{23} \\ b_{13} & b_{23} & 1
		\end{pmatrix})$ & 5
		\\ \hline
		$(\begin{pmatrix}
		1 & 0 & 0 \\ 0& 0& 0 \\ 0 & 0 & 0
		\end{pmatrix}, \begin{pmatrix}
		0 & 0 & 0 \\ 0& 1& 0 \\ 0 & 0 & 0
		\end{pmatrix})$ & $1|3|2$
		& $(\begin{pmatrix}
		1 & a_{12} & a_{13} \\ a_{12} & a_{22} & a_{23} \\ a_{13} & a_{23} & a_{33}
		\end{pmatrix}, \begin{pmatrix}
		b_{11} & b_{12} & 0 \\ b_{12} & 1 & 0 \\ 0 & 0 & 0
		\end{pmatrix})$ & 3
		\\ \hline
		$(\begin{pmatrix}
		0 & 0 & 0 \\ 0& 1 & 0 \\ 0 & 0 & 0
		\end{pmatrix}, \begin{pmatrix}
		0 & 0 & 0 \\ 0& 0 & 0 \\ 0 & 0 & 1
		\end{pmatrix})$ & $2|1|3$
		& $(\begin{pmatrix}
		0 & 0 & 0 \\ 0 & 1 & a_{23} \\ 0 & a_{23} & a_{33}
		\end{pmatrix}, \begin{pmatrix}
		b_{11} & b_{12} & b_{13} \\ b_{12} & b_{22} & b_{23} \\ b_{13} & b_{23} & 1
		\end{pmatrix})$ &3
		\\ \hline
		$(\begin{pmatrix}
		0 & 0 & 0 \\ 0& 1& 0 \\ 0 & 0 & 0
		\end{pmatrix}, \begin{pmatrix}
		1 & 0 & 0 \\ 0& 0& 0 \\ 0 & 0 & 0
		\end{pmatrix})$ & $2|3|1$
		& $(\begin{pmatrix}
		0 & 0 & 0 \\ 0 & 1 & a_{23} \\ 0 & a_{23} & a_{33}
		\end{pmatrix}, \begin{pmatrix}
		1 & 0 & 0 \\ 0& 0 & 0 \\ 0 & 0 & 0
		\end{pmatrix})$ & 2
		 \\ \hline
		$(\begin{pmatrix}
		0 & 0 & 0 \\ 0& 0& 0 \\ 0 & 0 & 1
		\end{pmatrix}, \begin{pmatrix}
		0 & 0 & 0 \\ 0& 1& 0 \\ 0 & 0 & 0
		\end{pmatrix})$ & $3|1|2$
		& $(\begin{pmatrix}
		0 & 0 & 0 \\ 0& 0& 0 \\ 0 & 0 & 1
		\end{pmatrix}, \begin{pmatrix}
		b_{11} & b_{12} & 0 \\ b_{12} & 1 & 0 \\ 0 & 0 & 0
		\end{pmatrix}
		)$ & 2
		 \\ \hline
		$(\begin{pmatrix}
		0 & 0 & 0 \\ 0& 0& 0 \\ 0 & 0 & 1
		\end{pmatrix}, \begin{pmatrix}
		1 & 0 & 0 \\ 0& 0& 0 \\ 0 & 0 & 0
		\end{pmatrix})$ & $3|2|1$ &
		$(\begin{pmatrix}
			0 & 0 & 0 \\ 0& 0& 0 \\ 0 & 0 & 1
		\end{pmatrix}, \begin{pmatrix}
			1 & 0 & 0 \\ 0& 0& 0 \\ 0 & 0 & 0
		\end{pmatrix})$ &0
		 \\ \hline
		$(\begin{pmatrix}
		0 & 1 & 0 \\ 1& 0& 0 \\ 0 & 0 & 0
		\end{pmatrix}, \begin{pmatrix}
		0 & 0 & 0 \\ 0& 0& 0 \\ 0 & 0 & 1
		\end{pmatrix})$ & $12|3$ &
		$(\begin{pmatrix}
		0 & 1 & a_{13} \\ 1 & a_{22} & a_{23} \\ a_{13} & a_{23} & a_{33}
		\end{pmatrix}, \begin{pmatrix}
		b_{11} & b_{12} & b_{13} \\ b_{12} & b_{22} & b_{23} \\ b_{13} & b_{23} & 1
		\end{pmatrix})$ &4
		\\ \hline
		$(\begin{pmatrix}
		0 & 0 & 1 \\ 0& 0& 0 \\ 1 & 0 & 0
		\end{pmatrix}, \begin{pmatrix}
		0 & 0 & 0 \\ 0& 1& 0 \\ 0 & 0 & 0
		\end{pmatrix})$ & $13|2$ &
		$(\begin{pmatrix}
		0 & 0 & 1 \\ 0 & 0 & a_{23} \\ 1 & a_{23} & a_{33}
		\end{pmatrix}, \begin{pmatrix}
		b_{11} & b_{12} & 0 \\ b_{12} & 1 & 0 \\ 0 & 0 & 0
		\end{pmatrix})$ & 2
		 \\ \hline
		$(\begin{pmatrix}
		0 & 0 & 0 \\ 0& 0 & 1 \\ 0 & 1 & 0
		\end{pmatrix}, \begin{pmatrix}
		1 & 0 & 0 \\ 0& 0 & 0 \\ 0 & 0 & 0
		\end{pmatrix})$ & $23|1$ &
		$(\begin{pmatrix}
			0 & 0 & 0 \\ 0& 0 & 1 \\ 0 & 1 & a_{33}
		\end{pmatrix}, \begin{pmatrix}
			1 & 0 & 0 \\ 0& 0 & 0 \\ 0 & 0 & 0
		\end{pmatrix})$ &1
		 \\ \hline
		$(\begin{pmatrix}
		1 & 0 & 0 \\ 0& 0& 0 \\ 0 & 0 & 0
		\end{pmatrix}, \begin{pmatrix}
		0 & 0 & 0 \\ 0& 0& 1 \\ 0 & 1 & 0
		\end{pmatrix})$ & $1|23$ &
		 $(\begin{pmatrix}
		1 & a_{12} & a_{13} \\ a_{12} & a_{22} & a_{23} \\ a_{13} & a_{23} & a_{33}
		\end{pmatrix}, \begin{pmatrix}
		b_{11} & b_{12} & b_{13} \\ b_{12} & b_{22} & 1 \\ b_{13} & 1 & 0
		\end{pmatrix})$ & 4\\ \hline
		$(\begin{pmatrix}
		0 & 0 & 0 \\ 0& 1& 0 \\ 0 & 0 & 0
		\end{pmatrix}, \begin{pmatrix}
		0 & 0 & 1 \\ 0& 0& 0 \\ 1 & 0 & 0
		\end{pmatrix})$ & $2|13$ 	&
		 $(\begin{pmatrix}
		0 & 0 & 0 \\ 0 & 1 & a_{23} \\ 0 & a_{23} & a_{33}
		\end{pmatrix}, \begin{pmatrix}
		b_{11} & b_{12} & 1 \\ b_{12} & b_{22} & 0 \\ 1 & 0 & 0
		\end{pmatrix})$ & 3
		 \\ \hline
		$(\begin{pmatrix}
		0 & 0 & 0 \\ 0& 0& 0 \\ 0 & 0 & 1
		\end{pmatrix}, \begin{pmatrix}
		0 & 1 & 0 \\ 1& 0& 0 \\ 0 & 0 & 0
		\end{pmatrix})$ & $3|12$ &
		$(\begin{pmatrix}
			0 & 0 & 0 \\ 0& 0& 0 \\ 0 & 0 & 1
		\end{pmatrix}, \begin{pmatrix}
			b_{11} & 1 & 0 \\ 1& 0& 0 \\ 0 & 0 & 0
		\end{pmatrix})$ & 1 \\ \hline
	\end{tabular}
\vspace{0.4em}
	\caption{Bia{\l}ynicki-Birula decomposition of $\CQ_3$.}
\label{TheTable}
\end{table}
\end{example}

\begin{remark}
Describing the ring structure on $A^\bullet(\CQ_n)$ is considerably more difficult, but can be done. See for instance \cite{DeConciniGoreskyMacPhersonProcesi}.
\end{remark}

 \section{Bodensee Program}\label{sec:Bodensee}

The answers to classical problems in enumerative geometry are simply integers, e.g.~there are \emph{two} plane quadrics that pass through four general points in $\PP^2$ and are tangent to one general line. However, such problems often come in discrete families. Above, we may differ the dimension of the space $\PP^{n-1}$ and the number of points $a$. Note that if we want the answer to be different from zero and infinity, then the number of general tangent hyperplanes needs to be equal to $\binom{n+1}{2}-1-a$. Thus, we may ask about properties of the function $\bar{\phi}(n,a)$ that counts quadrics in $\PP^{n-1}$ that pass through $a$ general points and are tangent to $\binom{n+1}{2}-1-a$ general hyperplanes. 
Note that by \Cref{prop:muGivesNumberOfQuadrics}, we have $\bar{\phi}(n,a)=\mu_a(S^2\CC^n)=\phi(n,a+1)$.

Here, we would like to underline the first shift of interest: instead of asking ``how many" we ask for ``properties of the function that answers the how many question". In the example above, due to the nature of the tools used to compute $\bar{\phi}(n,a)$ it is natural to consider $n$ fixed. Indeed, then we obtain one variety of complete quadrics $\CQ(\CC^{n})$. The Chow ring of this variety encodes all of the numbers $\bar{\phi}(n,a)$. This perspective is classical, but less interesting from our point of view. Indeed, going to the more sophisticated theorems about cohomology rings, we may deduce properties of the numbers $\bar{\phi}(n,\cdot)$, like log-concavity. Still, it is only a finite sequence of integers, for each fixed $n$, as we must have $0\leq a\leq \binom{n+1}{2}-1$.

Inspired, not by algebraic geometry, but by other disciplines, like algebraic statistics, we change the question and ask for properties of $\bar{\phi}(n,a)$, as a function of $n$ for any fixed $a$. This seems a bad idea, as now we do not have one variety on which we may do intersection theory. Instead, the variety changes, when the argument $n$ does. Surprisingly, $\bar{\phi}(\cdot, a)$ is always a polynomial!

Let us present two easy examples.
\begin{example}\label{exm:diag}
Let $\CC^a$ be a general linear subspace of $\CC^n$. What is the degree of the reciprocal variety obtained as the image of $\CC^a$ by the rational map inverting all the coordinates?

We note that after identifying $\CC^n$ with the space of $n\times n$ diagonal matrices, this is precisely the setting of matrix inversion described in \Cref{subs:permutvar,sec:perCohomology}.

The answer is now a function $d(a,n)$, where $1\leq a\leq n$. We have $d(a,n)=\binom{n-1}{a-1}$. Of course, for fixed $n$ this is just a finite sequence of numbers. For fixed $a$ we obtain a polynomial. Clearly, for all $n$ and $a$ we obtain the Pascal triangle. These numbers are also the beta invariants of uniform matroids, see~\cite{sturmfels2010multivariate}.
\end{example}
\begin{example}\label{exm:sym}
	
Passing from diagonal matrices in \Cref{exm:diag} to symmetric matrices we have the following question:

Let $\CC^a$ be a general linear subspace of the space of $n\times n$ symmetric matrices. What is the degree of the variety obtained as the image of $\CC^a$ by the rational map inverting the matrices?

The answer is precisely $\phi(n,a)$, where $\phi$ is the function introduced in Section~\ref{sec:polynomiality} and recalled at the beginning of this section. Evaluating at different  $n$'s and $a$'s we obtain:

\begin{tabular}{llllllllll}
1 & 2 &4& 4 &2 & 1 &0& 0& 0 & $\cdots$\\
1 & 3 &9& 17 &21 & 21 &17& 9 & 3 & $\cdots$\\
1 & 4 &16& 44 &86 & 137 &188& 212 &188& $\cdots$
\end{tabular}

This is exactly \cite[Table 1]{sturmfels2010multivariate}.
\end{example}

With this review, we would like to initiate the \emph{Bodensee program}. Its aim is precisely to understand the functions, that are answers to natural enumerative problems, that come in sequences. In particular, we would like to understand characteristic numbers, as functions of parameters. For a general introduction of chromatic numbers for tensors we refer to the recent article \cite{jaAustin}.

Let us present below a few recent theorems that are in this spirit.
\begin{theorem}[{\cite{REACT}}]
Let $d=5$ or $d\geq 7$ and $b<n(d-2)+3$. The number of degree $d$ (smooth) hypersurfaces in $\PP^n$ tangent to $b$ general hyperplanes and going through $\binom{n+d}{n}-1-b$ general points equals $(n(d-1)^{n-1})^b$.
\end{theorem}
\begin{theorem}[{\cite{MMMSV}}]\label{thm:polygen}
Consider a general $b$-dimensional subspace $L$ of general\slash symmetric\slash skew-symmetric $n\times n$ matrices. Let $f(b,n)$ be the degree of the variety dual to $L\cap(\det=0)$. Then $f(b,\cdot)$ is a polynomial.
\end{theorem}
The formulas for the leading coefficient are provided in \cite{borzi2021leading}. In the case of symmetric matrices, the number $f(b,n)$ is equal to the algebraic degree of semidefinite programming $\sum_r \delta(b,n,r)$ and Theorem~\ref{thm:polygen} agrees with Theorem~\ref{thm:polyDelta}.

Below we present open problems that we consider part of Bodensee Program and find particularly interesting.
\subsection{Open Problems}
\begin{problem} Let $f(d,b,n)$ be the function that counts the number of degree $d$ (smooth) hypersurfaces in $\PP^n$ going through $b$ general points and tangent to $\binom{n+d}{n}-1-b$ general hyperplanes. Describe the function $f(d,b,\cdot)$. Can we do it at least for small $b$?
\end{problem}
For plane cubics, i.e. $d=3$ and $n=2$, the numbers $f(3,b,2)$ were first computed by Maillard~\cite{Maillard} and Zeuthen~\cite{Zeuthen}, and for a more modern treatment, see \cite{AluffiCubics}. For plane quartics, the characteristic numbers $f(4,b,2)$ were computed by Vakil~\cite{VakilCharacteristic}.
\begin{problem}
 Let $\phi(n,b)$ be the function that counts quadrics in $\PP^{n-1}$ that pass through $b$ general points and are tangent to $ \binom{n+1}{2}-1-b$ general hyperplanes. Are the coefficients of the polynomial $\phi(\cdot,b)$ log-concave? Can we provide a better understanding for the evaluations of $\phi(\cdot,b)$ at negative arguments?
 \end{problem}
\begin{problem}
	For any class in $Y \in A^d(\CQ_n)$, we can define a function
	\[
	\phi_Y(n) = \int_{\CQ_n}{Y \cdot L_1^{\binom{n+1}{2}-d-1}}.
	\]
	Is this always a polynomial in $n$, for large enough $n$?
\end{problem}
\begin{problem}
	\Cref{prop:EulerChar} states that the alternating sum of the numbers $\nu_i(L)$ has a geometric interpretation as an Euler characteristic. Is there a similar geometric interpretation of the alternating sum of the $\nu_i(L)$?
\end{problem}
\bibliography{MatroidsBib}{}
\bibliographystyle{plain}
\end{document}